\documentclass[11pt,reqno]{amsart}

\newtheorem{proposition}{Proposition}[section]
\newtheorem{lemma}[proposition]{Lemma}
\newtheorem{corollary}[proposition]{Corollary}
\newtheorem{theorem}[proposition]{Theorem}

\theoremstyle{definition}
\newtheorem{definition}[proposition]{Definition}
\newtheorem{example}[proposition]{Example}

\newtheorem{remark}[proposition]{Remark}

\newcommand{\thlabel}[1]{\label{th:#1}}
\newcommand{\thref}[1]{Theorem~\ref{th:#1}}
\newcommand{\selabel}[1]{\label{se:#1}}
\newcommand{\seref}[1]{Section~\ref{se:#1}}
\newcommand{\lelabel}[1]{\label{le:#1}}
\newcommand{\leref}[1]{Lemma~\ref{le:#1}}
\newcommand{\prlabel}[1]{\label{pr:#1}}
\newcommand{\prref}[1]{Proposition~\ref{pr:#1}}
\newcommand{\colabel}[1]{\label{co:#1}}
\newcommand{\coref}[1]{Corollary~\ref{co:#1}}
\newcommand{\relabel}[1]{\label{re:#1}}
\newcommand{\reref}[1]{Remark~\ref{re:#1}}
\newcommand{\exlabel}[1]{\label{ex:#1}}
\newcommand{\exref}[1]{Example~\ref{ex:#1}}
\newcommand{\delabel}[1]{\label{de:#1}}
\newcommand{\deref}[1]{Definition~\ref{de:#1}}
\newcommand{\eqlabel}[1]{\label{eq:#1}}
\newcommand{\equref}[1]{(\ref{eq:#1})}

\newcommand{\Cc}{\mathcal{C}}

\def\*C{{}^*\hspace*{-1pt}{\Cc}}
\def\text#1{{\rm {\rm #1}}}

\input xy
\xyoption {all} \CompileMatrices

\usepackage{amssymb}
\usepackage{color,amssymb,graphicx,amscd,amsmath}
\usepackage[colorlinks,urlcolor=blue,linkcolor=blue,citecolor=blue]{hyperref}

\begin{document}
\title[The global extension problem for Leibniz algebras]
{The global extension problem, co-flag and metabelian Leibniz
algebras}

\author{G. Militaru}
\address{Faculty of Mathematics and Computer Science, University of Bucharest, Str.
Academiei 14, RO-010014 Bucharest 1, Romania}
\email{gigel.militaru@fmi.unibuc.ro and gigel.militaru@gmail.com}


\thanks{This work was supported by a grant of the Romanian National
Authority for Scientific Research, CNCS-UEFISCDI, grant no.
88/05.10.2011.}



\begin{abstract}
Let $\mathfrak{L}$ be a Leibniz algebra, $E$ a vector space and
$\pi : E \to \mathfrak{L}$ an epimorphism of vector spaces with $
\mathfrak{g} = {\rm Ker} (\pi)$. The global extension problem asks
for the classification of all Leibniz algebra structures that can
be defined on $E$ such that $\pi : E \to \mathfrak{L}$ is a
morphism of Leibniz algebras: from a geometrical viewpoint this means
to give the decomposition of the groupoid of all such structures in its connected
components and to indicate a point in each component. All such Leibniz algebra
structures on $E$ are classified by a global cohomological
object ${\mathbb G} {\mathbb H} {\mathbb L}^{2} \, (\mathfrak{L}, \, \mathfrak{g})$
which is explicitly constructed. It is shown that ${\mathbb G} {\mathbb H} {\mathbb L}^{2} \,
(\mathfrak{L}, \, \mathfrak{g})$ is the coproduct of all local
cohomological objects $ {\mathbb H} {\mathbb L}^{2} \,  \,
(\mathfrak{L}, \, (\mathfrak{g}, [-,-]_{\mathfrak{g}}))$ that are
classifying sets for all extensions of $\mathfrak{L}$
by all Leibniz algebra structures $(\mathfrak{g},
[-,-]_{\mathfrak{g}})$ on $\mathfrak{g}$. The
second cohomology group ${\rm HL}^2 \, (\mathfrak{L}, \, \mathfrak{g})$ of Loday and Pirashvili appears
as the most elementary piece among all components of ${\mathbb G} {\mathbb H} {\mathbb L}^{2} \,
(\mathfrak{L}, \, \mathfrak{g})$. Several examples are
worked out in details for co-flag Leibniz algebras over
$\mathfrak{L}$, i.e. Leibniz algebras $\mathfrak{h}$ that have a
finite chain of epimorphisms of Leibniz algebras $\mathfrak{L}_n :
= \mathfrak{h} \stackrel{\pi_{n}}{\longrightarrow}
\mathfrak{L}_{n-1} \, \cdots \, \mathfrak{L}_1 \stackrel{\pi_{1}}
{\longrightarrow} \mathfrak{L}_{0} := \mathfrak{L}$ such that
${\rm dim} ( {\rm Ker} (\pi_{i}) ) = 1$, for all $i = 1, \cdots,
n$. Metabelian Leibniz algebras are introduced, described and classified
using pure linear algebra tools.
\end{abstract}

\maketitle

\section*{Introduction}
Introduced by Bloh \cite{Bl1} and rediscovered by Loday
\cite{Lod2} the study of Leibniz algebras has met with an
explosion of interest in the last twenty years after the papers
\cite{cu, Lod2, LoP} that set the foundations of the theory were
published. There are two reasons for this interest. The first one
is purely theoretical: Leibniz algebras are non-commutative
generalizations of Lie algebras and hence several classical
theorems or larger topics for study (the classification problem,
the cohomology theory) from Lie algebras were extended to Leibniz
algebras. The second and strongest reason consists of the
interactions and applications of Leibniz algebras in several areas
of mathematics and mathematical physics such as: classical or
non-commutative differential geometry, vertex operator algebras,
the Godbillon-Vey invariants for foliations, integrable systems,
representation theory, etc. For additional explanations and
motivations we refer to \cite{AOR, AAO, Bar, ACM, cam, CK, CC,
CLOK, CI, casas, FM, Go, hu, KW2000, kosm, mas, RHa, RHb} and the
references therein.

Recently \cite{am-2013c} the following problem called the \emph{extending structures problem},
was addressed at the level of Leibniz algebras generalizing the one from Lie algebras
\cite{am-2013b} (see also \cite{am-2011, am-2010, am-2013a}). Equivalently restated, it consists
of the following question: let $\mathfrak{g}$ be a Leibniz  algebra, $E$ a vector space
and $i: \mathfrak{g} \to E$ a linear injective map.
Describe and classify the set of all Leibniz algebra structures $[-, -]_E$ that can be defined on $E$
such that $i: \mathfrak{g} \to (E, \, [-, -]_E)$
is a morphism of Leibniz algebras. It becomes tempting and natural to look at the categorical dual
of it by just reversing the arrow from the input data of the extending structures problem.
We arrive at the following question that is the subject of this paper:

\textbf{The global extension problem.} \textit{
Let $\mathfrak{L}$ be a Leibniz algebra, $E$ a vector space and
$\pi : E \to \mathfrak{L}$ an epimorphism of vector spaces.
Describe and classify the set of all Leibniz algebra structures
$[-, -]_E$ that can be defined on $E$
such that $\pi : (E, \, [-, -]_E ) \to \mathfrak{L}$ is a
morphism of Leibniz algebras.}

If $\mathfrak{L} = \{0\}$, then the GE-problem asks for the classification of all
Leibniz algebra structures on an arbitrary vector space $E$. This is one of the problems that have been extensively
studied during the last years and it is very difficult for vector spaces
of large dimension: the classification of all Leibniz algebras is known up
to dimension $4$ and this result was obtained recently \cite{CK}.
For this reason, from now on we will assume that $\mathfrak{L} \neq \{0\}$.
Let us explain the significance of the word \emph{'global'} from the
above question and make the connection with the classical extension problem that
was approached in \cite{LoP} only for the abelian case.
Let $\mathfrak{g}$ and $\mathfrak{L}$ be two given Leibniz
algebras. The extension problem asks for the classification of all
extensions of $\mathfrak{L}$ by $\mathfrak{g}$, i.e. of all
Leibniz algebras $\mathfrak{E}$ that fit into an exact sequence
\begin{eqnarray*}
\xymatrix{ 0 \ar[r] & \mathfrak{g} \ar[r]^{i} & \mathfrak{E}
\ar[r]^{\pi} & \mathfrak{L} \ar[r] & 0 }
\end{eqnarray*}
The classification is up to an isomorphism of Leibniz algebras
that stabilizes $\mathfrak{g}$ and co-stabilizes $\mathfrak{L}$
and we denote by ${\mathcal E} \,  (\mathfrak{L}, \,
\mathfrak{g})$ the isomorphism classes of all extensions of
$\mathfrak{L}$ by $\mathfrak{g}$ up to this equivalence relation.
If $\mathfrak{g}$ is abelian, then ${\mathcal E} \, (\mathfrak{L},
\, \mathfrak{g}) \cong {\rm HL}^2 (\mathfrak{L}, \,
\mathfrak{g})$, where ${\rm HL}^2 (\mathfrak{L}, \, \mathfrak{g})$
is the second cohomology group \cite[Proposition 1.9]{LoP}. Now,
if $(E, \, [-, -]_E)$ is a Leibniz algebra structure on $E$ such
that $\pi: (E, \, [-, -]_E) \to \mathfrak{L}$ is a morphism of
Leibniz algebras, then $(E, \, [-, -]_E)$ is an extension of
$\mathfrak{L}$ by $\mathfrak{g}: = {\rm Ker} (\pi)$, which is a
Leibniz subalgebra of $(E, \, [-, -]_E)$. But this Leibniz algebra
structure on $\mathfrak{g}$ is not fixed from the input data as in
the case of the extension problem: it depends essentially on the
Leibniz algebra structures on $E$ which we are looking for. Thus,
we can conclude that the classical extension problem is the
\emph{'local'} version of the GE-problem: namely, the one in which
the Leibniz algebra structure on ${\rm Ker} (\pi)$ is fixed. In
other words, the GE-problem can be viewed as a \emph{dynamical}
extension problem since the structures of Leibniz algebras on
$\mathfrak{g}$ are not fixed. As in the classical extension
problem, by classification of Leibniz algebra structures on $E$ we
mean the classification up to an isomorphism of Leibniz algebras
$(E, \, [- , - ]_E) \cong (E, \, [- , - ]^{'}_E)$ that stabilizes
$\mathfrak{g} = {\rm Ker} (\pi)$ and co-stabilizes $\mathfrak{L}$.
In this case we shall say that the Leibniz algebra structures
$\{-, \, -\}_E$ and $\{-, \, -\}^{'}_E $ on $E$ are
\emph{cohomologous} and we denote this by $(E, \{-, \, -\}_E)
\approx (E, \{-, \, -\}^{'}_E)$. Let ${\rm Gext} (\mathfrak{L}, \,
E)$ be the isomorphism classes of all Leibniz algebra structures
on $E$ such that $\pi: E \to \mathfrak{L}$ is a morphism of
Leibniz algebras via this identification. The answer to the
GE-problem will be given if we parameterize and explicitly compute
${\rm Gext} (\mathfrak{L}, \, E)$. A more conceptual explanation
is given in \seref{prel}: the set of all Leibniz algebra
structures $\{-, \, -\}_E$ that can be defined on the vector space
$E$ such that $\pi : (E, \, \{-, \, -\}_E) \to \mathfrak{L}$ is a
morphism of Leibniz algebras is a groupoid, denoted by ${\mathcal
C} (E, \, \mathfrak{L})$, and the answer to the GE-problem from a
geometrical viewpoint means to give the decomposition of this
groupoid in its connected components and to indicate a point in
each component.

The paper is organized as follows: In \seref{prel} we recall the basic concept and the
definition of the crossed products of Leibniz algebras.
\seref{globalex} contains the main results of the paper. \thref{extensioliecros} shows that
any Leibniz algebra structure $(E, \, [-, \, -]_E)$ such that $\pi : (E, [-, \, -]_E) \to
\mathfrak{L}$ is a morphism of Leibniz algebras
is cohomologous to a crossed product. Based on this,
\thref{main1222} gives the theoretical answer to the
GE-problem by explicitly constructing a
global non-abelian cohomological type object, denoted by ${\mathbb G} {\mathbb H} {\mathbb
L}^{2} \, (\mathfrak{L}, \, \mathfrak{g})$, which will
parameterize ${\rm Gext} \, (\mathfrak{L},
\, E)$. The relation between the GE-problem and the classical (i.e. local)
extension problem for Leibniz algebras and the decomposition of ${\mathbb G} {\mathbb H} {\mathbb
L}^{2} \, (\mathfrak{L}, \, \mathfrak{g})$ as the coproduct of all local
cohomological 'groups' $ {\mathbb H} {\mathbb L}^{2} \,  \,
(\mathfrak{L}, \, (\mathfrak{g}, [-,-]_{\mathfrak{g}}))$ is given in
\coref{desccompcon}. We mention that ${\mathbb H} {\mathbb L}^{2} \,  \,
(\mathfrak{L}, \, (\mathfrak{g}, [-,-]_{\mathfrak{g}}))$ are
classifying sets for all extensions of $\mathfrak{L}$
by $(\mathfrak{g}, [-,-]_{\mathfrak{g}})$. The second cohomology group
${\rm HL}^2 \, (\mathfrak{L}, \, \mathfrak{g})$ of Loday and Pirashvili \cite{LoP} is just
one piece, and the most elementarily described one at that, among all components of ${\mathbb H} {\mathbb L}^{2} \,  \,
(\mathfrak{L}, \, (\mathfrak{g}, [-,-]_{\mathfrak{g}}))$. More precisely,
${\rm HL}^2 \, (\mathfrak{L}, \, \mathfrak{g}) = {\mathbb H} {\mathbb L}^{2} \,  \,
(\mathfrak{L}, \, (\mathfrak{g}, [-,-]_{\mathfrak{g}} = 0 ))$. Metabelian
Leibniz algebras are introduced as crossed products of abelian Leibniz algebras: the explicit bracket
and their classification is proved in
\coref{cazuabspargere} in terms of linear algebra tools. An explicit example is given in
\coref{metadim3clas}: in particular, a class of metabelian $n+1$-dimensional Leibniz algebras are described and classified by
the set of matrices $(A, \, B, \,  \gamma) \in {\rm M}_n (k) \times {\rm M}_n (k) \times k^n$ satisfying the condition:
$AB = BA = - B^2$ and $B \gamma = 0$ (\coref{meta10}).
Computing the classifying object ${\mathbb G} {\mathbb H} {\mathbb
L}^{2} \, (\mathfrak{L}, \, \mathfrak{g})$ is a highly nontrivial problem:
recall that for $\mathfrak{L} := 0 $ this object parameterizes all isomorphism classes of Leibniz algebras that can be defined on a
given vector space $\mathfrak{g}$. In \seref{coflag}
we shall identify a way of computing this object for a class of what we have called
co-flag Leibinz algebras over $\mathfrak{L}$
as defined in \deref{coflaglbz}. All co-flag Leibniz algebras over $\mathfrak{L}$
can be completely described by a recursive
reasoning where the key step is the case when $\mathfrak{g} := {\rm Ker} (\pi)$ has
dimension $1$. This case is solved in \thref{clascoflagl} where ${\mathbb G} {\mathbb H} {\mathbb
L}^{2} \, (\mathfrak{L}, \, k)$ is completely described. Several examples are worked out in details
for co-flag Leibniz algebras as well as for metabelian Leibniz algebras.

\section{Preliminaries}\selabel{prel}
For two sets $X$ and $Y$ we shall denote by $X \sqcup Y$ the
coproduct in the category of sets of $X$ and $Y$, i.e. $X \sqcup
Y$ is the disjoint union of $X$ and $Y$. All vector spaces,
Leibniz algebras, linear or bilinear maps are over an arbitrary
field $k$. A map $f: V \to W$ between two vector spaces is called
the \emph{trivial} map if $f (v) = 0$, for all $v\in V$. A Leibniz
algebra is a vector space $\mathfrak{L}$ with a bilinear map $[- ,
\, -] : \mathfrak{L} \times \mathfrak{L} \to \mathfrak{L}$
satisfying the Leibniz law:
\begin{equation}\eqlabel{lz}
[x, \, [y, \, z] \, ] = [ \, [x, \, y], \, z ] - [ \,[x, \, z], \,
y ]
\end{equation}
for all $x$, $y$, $z\in \mathfrak{L}$. Any Lie algebra is a
Leibniz algebra and a Leibniz algebra satisfying $[x, \, x] = 0$,
for all $x \in \mathfrak{L}$ is a Lie algebra. Any vector space
$\mathfrak{L}$ has a structure of Leibniz algebra with the
trivial bracket: $[x, \, y] = 0$, for all $x$, $y \in
\mathfrak{L}$. Such a Leibniz algebra is called \emph{abelian}
and will be denoted throughout the paper by
$\mathfrak{L}_0$. For several examples of Leibniz algebras we refer to \cite{Lod2},
\cite{LoP}. ${\rm Der} (\mathfrak{L})$ (resp. ${\rm ADer} (\mathfrak{L})$) stands for the
space of all derivations (resp. anti-derivations \cite{Lod2},
\cite{am-2013c}) of a Leibniz algebra $\mathfrak{L}$, that is, all linear maps
$\Delta: \mathfrak{L} \to \mathfrak{L}$ (resp. $D : \mathfrak{L}
\to \mathfrak{L}$) such that for any $x$, $y \in \mathfrak{L}$ we
have:
$$
\Delta ([x, \, y]) = [\Delta(x), \, y] + [x, \, \Delta(y)], \quad
\bigl( \, {\rm resp.} \,\,\,  D ([x, \, y]) = [D(x), \, y] - [D(y), \, x] \, \bigl)
$$
Let $\mathfrak{L}$ be a given Leibniz algebra, $E$ a vector space and
$\pi : E \to \mathfrak{L}$ an epimorphism of vector spaces with $
\mathfrak{g} := {\rm Ker} (\pi)$. We define a small category
${\mathcal C} (E, \, \mathfrak{L})$ as follows: the objects of
${\mathcal C} (E, \, \mathfrak{L})$ are Leibniz algebra structures
$\{-, \, -\}_E$ on the vector space $E$ such that
$\pi : (E, \, \{-, \, -\}_E) \to \mathfrak{L}$ is a morphism of
Leibniz algebras. A morphism $\varphi: \{-, \, -\}_E \to \{-, \,
-\}^{'}_E$ in  ${\mathcal C} (E, \, \mathfrak{L})$ is
a morphism of Leibniz algebras $\varphi: (E, \{-, \, -\}_E) \to (E, \{-, \,
-\}^{'}_E)$ which stabilizes $\mathfrak{g}$ and co-stabilizes
$\mathfrak{L}$, i.e. the following diagram
\begin{eqnarray} \eqlabel{diagrama}
\xymatrix {& \mathfrak{g} \ar[r]^{i} \ar[d]_{Id} & {E}
\ar[r]^{\pi} \ar[d]^{\varphi} & \mathfrak{L} \ar[d]^{Id}\\
& \mathfrak{g} \ar[r]^{i} & {E}\ar[r]^{\pi } & \mathfrak{L}}
\end{eqnarray}
is commutative. In this case we shall say that the Leibniz algebra structures $\{-, \, -\}_E$ and
$\{-, \, -\}^{'}_E \in {\mathcal C} (E, \, \mathfrak{L})$ are
\emph{cohomologous} and we denote this by
$(E, \{-, \, -\}_E) \approx (E, \{-, \, -\}^{'}_E)$. The category
${\mathcal C} (E, \, \mathfrak{L})$ is a groupoid, i.e. any morphism is
an isomorphism - this fact can be proven as an elementary exercise using only
the above definition (an alternative and more conceptual
proof is given in \leref{HHH}). In particular, we obtain that $\approx$ is an
equivalence relation on the set of objects of ${\mathcal C} (E, \, \mathfrak{L})$ and we
denote by ${\rm Gext} \, (E, \, \mathfrak{L})$ the set of all equivalence classes via
$\approx$, i.e. ${\rm Gext} \, (E, \, \mathfrak{L}) := {\mathcal C} (E, \, \mathfrak{L})/\approx$.
${\rm Gext} \, (E, \, \mathfrak{L})$ is the classifying
object of the global extension problem: the theoretical answer, from an algebraic view point,
to the GE-problem will be given when we explicitly compute ${\rm Gext} \,
(E, \, \mathfrak{L})$ for a given Leibniz algebra $\mathfrak{L}$
and a vector space $E$. Of course, ${\rm Gext} \, (E, \,
\mathfrak{L})$ is empty if $\mathfrak{L}$ is finite dimensional
and ${\rm dim} (E) < {\rm dim} (\mathfrak{L})$. From a
geometrical view point the answer to the GE-problems means to give the
decomposition of the groupoid ${\mathcal C} (E, \, \mathfrak{L})$ in its connected
components and to indicate a 'point' in each such component.

\subsection*{Crossed products for Leibniz algebras.}
We recall the construction of crossed products of Leibniz algebras following the
terminology of \cite[Section 4]{am-2013c}.

\begin{definition} \delabel{crosslbz}
Let $\mathfrak{L} = (\mathfrak{L}, \, [-,\, -])$ be a Leibniz
algebra and $\mathfrak{g}$ a vector space. A \emph{pre-crossed
data} of $\mathfrak{L}$ by $\mathfrak{g}$ is a system $\Lambda
(\mathfrak{L}, \, \mathfrak{g}) = \bigl(\triangleleft, \,
\triangleright, \, f, \, [-, \, -]_{\mathfrak{g}} \bigl)$
consisting of four bilinear maps
$$
\triangleleft : \mathfrak{g} \times \mathfrak{L} \to \mathfrak{g},
\quad \triangleright: \mathfrak{L} \times \mathfrak{g} \to
\mathfrak{g}, \quad f: \mathfrak{L} \times \mathfrak{L} \to
\mathfrak{g}, \quad [-, \, -]_{\mathfrak{g}} \, : \mathfrak{g}
\times \mathfrak{g} \to \mathfrak{g}
$$
Let $\Lambda (\mathfrak{L}, \, \mathfrak{g}) =
\bigl(\triangleleft, \, \triangleright, \, f, \, [-, \,
-]_{\mathfrak{g}} \bigl)$ be a pre-crossed data of $\mathfrak{L}$
by $\mathfrak{g}$ and we denote by $ \mathfrak{g} \, \# \,
\mathfrak{L} = \mathfrak{g} \, \#_{\Lambda (\mathfrak{L}, \,
\mathfrak{g})} \, \mathfrak{L}$ the vector space $\mathfrak{g}
\times \, \mathfrak{L} $ with the bracket $\{-, \, - \}$ given for
any $g$, $h \in \mathfrak{g}$ and $x$, $y \in \mathfrak{L}$ by:
\begin{equation}\eqlabel{brackcrosspr}
\{(g, x), \, (h, y)\} := \bigl( [g, \, h]_{\mathfrak{g}} + x
\triangleright h + g \triangleleft y + f(x, y), \, [x, \, y ]
\bigl)
\end{equation}
Then $\mathfrak{g} \, \# \, \mathfrak{L}$ is called the
\emph{crossed product} associated to $\Lambda (\mathfrak{L}, \,
\mathfrak{g}) = \bigl(\triangleleft, \, \triangleright, \, f, \,
[-, \, -]_{\mathfrak{g}} \bigl)$ if it is a Leibniz algebra with
the bracket \equref{brackcrosspr}. In this case the pre-crossed
data $\Lambda (\mathfrak{L}, \, \mathfrak{g}) =
\bigl(\triangleleft, \, \triangleright, \, f, \, [-, \,
-]_{\mathfrak{g}} \bigl)$ is called a \emph{crossed system} of
$\mathfrak{L}$ by $\mathfrak{g}$ and we denote by ${\mathcal C}
{\mathcal S} (\mathfrak{L}, \, \mathfrak{g})$ the set of all
crossed systems of the Leibniz algebra $\mathfrak{L}$ by
$\mathfrak{g}$.
\end{definition}

The bracket \equref{brackcrosspr} generalizes the one introduced in
\cite[Section 1.7]{LoP}, which is recovered as a special case if
$[-, \, -]_{\mathfrak{g}} = 0$. The set of necessary and sufficient axioms that
need to be satisfied by a pre-crossed datum
$\Lambda (\mathfrak{L}, \, \mathfrak{g}) = \bigl(\triangleleft, \,
\triangleright, \, f, \, [-, \, -]_{\mathfrak{g}} \bigl)$ such
that $\mathfrak{g} \, \# \, \mathfrak{L}$ is a crossed product was
highlighted in \cite[Section 4]{am-2013c},
derived as a special case of \cite[Theorem 2.3]{am-2013c}. More
precisely, we have the following:\footnote{The next
proposition can be also proven directly by testing the Leibniz law for the bracket
\equref{brackcrosspr} in all points of the form $(g, 0)$ and $(0,
x)$, for all $g \in \mathfrak{g}$, $x \in \mathfrak{L}$.}

\begin{proposition}\prlabel{axiocross}
A pre-crossed data $\Lambda (\mathfrak{L}, \, \mathfrak{g}) =
\bigl(\triangleleft, \, \triangleright, \, f, \, [-, \,
-]_{\mathfrak{g}} \bigl)$ of a Leibniz algebra $\mathfrak{L}$ by a
vector space $\mathfrak{g}$ is a crossed system if and only if
the following compatibilities
hold for any $g$, $h \in \mathfrak{g}$ and $x$, $y$, $z \in
\mathfrak{L}$:
\begin{enumerate}
\item[(CS0)] $(\mathfrak{g}, \, [-, \, -]_{\mathfrak{g}})$ is a
Leibniz algebra;

\item[(CS1)] $[g, \, h]_{\mathfrak{g}} \, \triangleleft x = [g, \,
h \triangleleft x]_{\mathfrak{g}} + [g \triangleleft x, \,
h]_{\mathfrak{g}}$;

\item[(CS2)] $g \triangleleft [x, \, y] = (g \triangleleft x)
\triangleleft y - (g \triangleleft y) \triangleleft x - [g, \,
f(x, y)]_{\mathfrak{g}}$;

\item[(CS3)] $x \rhd f(y, \, z) = f(x, \, y) \triangleleft z -
f(x, \, z) \triangleleft y + f([x, \, y], \, z) - f([x, \, z], \,
y) - f(x, \, [y, \, z])$;

\item[(CS4)] $x \rhd [g, \, h]_{\mathfrak{g}} = [x \rhd g, \,
h]_{\mathfrak{g}} - [x \rhd h, \, g]_{\mathfrak{g}}$;

\item[(CS5)] $[x, \, y] \rhd g = x \rhd (y \rhd g) + (x \rhd g)
\triangleleft y - [f(x, \, y), \, g]_{\mathfrak{g}}$;

\item[(CS6)] $[g, \, h \triangleleft x]_{\mathfrak{g}} + [g, \, x
\rhd h]_{\mathfrak{g}} = 0$;

\item[(CS7)] $x \rhd (y \rhd g) + x \rhd (g \triangleleft y) = 0$.
\end{enumerate}
\end{proposition}

From now on a crossed system of $\mathfrak{L}$ by $\mathfrak{g}$
will be viewed as a system of bilinear maps $\Lambda
(\mathfrak{L}, \, \mathfrak{g}) = \bigl(\triangleleft, \,
\triangleright, \, f, \, [-, \, -]_{\mathfrak{g}} \bigl)$
satisfying the axioms (CS0)-(CS7). We also use the following
convention during the paper: if one of the maps $\triangleleft$,
$\triangleright$, $f$, $[-, \, -]_{\mathfrak{g}}$ of a crossed
system is trivial then we will omit it from the system $\Lambda
(\mathfrak{L}, \, \mathfrak{g}) = \bigl(\triangleleft, \,
\triangleright, \, f, \, [-, \, -]_{\mathfrak{g}} \bigl)$. The set
${\mathcal C} {\mathcal S} (\mathfrak{L}, \, \mathfrak{g})$ of all
crossed systems of $\mathfrak{L}$ by $\mathfrak{g}$ is non-empty:
it contains the crossed system all maps of which are trivial.
The crossed product associated to it is just the direct
product $\mathfrak{g}_0 \times \mathfrak{L}$, between the abelian
Leibniz algebra $\mathfrak{g}_0$ with
$\mathfrak{L}$.

\begin{remark}
The axioms (CS0)-(CS7) are in fact quite natural although they look rather complicated at first sight
and can be interpreted as follows:
(CS3) is a non-abelian $2$-cocycle condition \cite{LoP}. Now, for any $x \in \mathfrak{L}$ we
denote by $\Delta_x = \, - \,  \triangleleft \, x$ (resp. $D_x = x \,
\triangleright \, - \, $) the linear maps from $\mathfrak{g}$ to
$\mathfrak{g}$ defined by $\Delta_x (g) := g \triangleleft x$
(resp. $D_x (g) := x \triangleright g$), for all $g \in
\mathfrak{g}$. The axiom (CS1) (resp. (CS4)) is equivalent to the fact that
$\Delta_x : \mathfrak{g} \to \mathfrak{g}$ (resp. $D_x : \mathfrak{g} \to \mathfrak{g}$) is a
derivation (resp. an anti-derivation) of the Leibniz algebra
$(\mathfrak{g}, \, [-, \, -]_{\mathfrak{g}})$, for any $x \in
\mathfrak{L}$. Finally, the axioms (CS2), (CS5), (CS6) and
respectively (CS7) are equivalent to the following commutating type
relations:
\begin{eqnarray*}
&& \Delta_y \circ \Delta_x - \Delta_x \circ \Delta_y = \Delta_{[x,
\, y]} + [-, \, f(x, \, y)]_{\mathfrak{g}} \eqlabel{comut1} \\
&& \Delta_y \circ D_x - D_x \circ \Delta_y = D_{[x, \, y]} +
[f(x, \, y), \, - ]_{\mathfrak{g}} \eqlabel{comut2} \\
&& \Delta_x (h) + D_x (h) \in {\rm Z}^r (\mathfrak{g}), \quad \, D_x \circ D_y + D_x \circ \Delta_y = 0 \eqlabel{comut3}
\end{eqnarray*}
for all $x$, $y \in \mathfrak{L}$ and $h \in \mathfrak{g}$, where
${\rm Z}^r (\mathfrak{g}) := \{ g \in \mathfrak{g} \, | \, [g, \,
z]_{\mathfrak{g}} = 0, \,\,\, \forall \,\,\, z \in \mathfrak{g}
\}$ is the right center of the Leibniz algebra $(\mathfrak{g}, \,
[-, \, -]_{\mathfrak{g}})$.
\end{remark}

There is more to be said about the axioms (CS0)-(CS7): they are non-abelian
generalizations of representations of a Leibniz algebra and
their cocycles associated \cite{LoP}. We recall their definition,
presented in an equivalent form such that the axioms \equref{bim1}-\equref{bim3} below
are exactly (CS2), (CS5) and (CS7) written for the trivial bracket $[-, \, -]_{\mathfrak{g}} =
0$ on $\mathfrak{g}$.

\begin{definition}\delabel{lbzbimodule}
Let $\mathfrak{L}$ be a Leibniz algebra. A \emph{representation}
of the Leibniz algebra $\mathfrak{L}$ or \emph{Leibniz bimodule
over $\mathfrak{L}$} is a triple $(\mathfrak{g}, \, \triangleleft,
\, \triangleright)$ consisting of a vector space $\mathfrak{g}$
and two bilinear maps $\triangleleft : \mathfrak{g} \times
\mathfrak{L} \to \mathfrak{g}$ and $\triangleright: \mathfrak{L}
\times \mathfrak{g} \to \mathfrak{g}$ satisfying the following
compatibility conditions for all $g \in\mathfrak{g}$ and $x$, $y
\in \mathfrak{L}$:
\begin{eqnarray}
g \triangleleft \left[x, \, y\right] &=& (g \triangleleft x)
\triangleleft y - (g \triangleleft y) \triangleleft x \eqlabel{bim1} \\
\left[x, \, y \right] \rhd g &=& x \rhd (y \rhd g) + (x \rhd g)
\triangleleft y  \eqlabel{bim2} \\
x \rhd (y \rhd g) &=& - \, x \rhd (g \triangleleft y)
\eqlabel{bim3}
\end{eqnarray}
We denote by ${}_{\mathfrak{L}}{\mathcal M}_{\mathfrak{L}}$ the
category of all Leibniz bimodules over $\mathfrak{L}$ with the
obvious maps, i.e. linear maps which are compatible with the left
and the right action of $\mathfrak{L}$.
\end{definition}

\begin{remark} \relabel{repLPvseu}
It is easy to see that the above axioms
\equref{bim1}-\equref{bim3} are equivalent to the compatibility conditions denoted
by (MLL), (LML) and (LLM) in \cite[Section 1.5]{LoP} that define a
representation of a Leibniz algebra $\mathfrak{L}$. Leibniz bimodules over
$\mathfrak{L}$ appear naturally in the description of crossed systems
of $\mathfrak{L}$ by a given abelian Leibniz algebra $\mathfrak{g}_0 = (\mathfrak{g}, \,
[-, \, -]_{\mathfrak{g}}:= 0)$. Indeed, a pre-crossed datum of the form
$\bigl(\triangleleft, \, \triangleright, \, f, \, [-, \, -]_{\mathfrak{g}} :=
0 \bigl)$ is a crossed system of if and only if $(\mathfrak{g}, \,
\triangleleft, \, \triangleright) \in {}_{\mathfrak{L}}{\mathcal
M}_{\mathfrak{L}}$ is a Leibniz bimodule over $\mathfrak{L}$ and
$f: \mathfrak{L} \times \mathfrak{L} \to \mathfrak{g}$ is a
$2$-cocyle, i.e. the compatibility condition (CS3) holds.
From now on, we denote by ${\mathcal C} {\mathcal S}_{0} (\mathfrak{L}, \, \mathfrak{g}_0)$ the set of all
\emph{abelian local crossed systems} of $\mathfrak{L}$ by
$\mathfrak{g}_0$, i.e. the set of all triples $\bigl(\triangleleft, \, \triangleright, \, f \bigl)$
such that $\bigl(\triangleleft, \, \triangleright, \, f, \, [-, \, -]_{\mathfrak{g}} :=
0 \bigl)$ is a crossed system.
\end{remark}

\section{The global extension problem for Leibniz algebras}\selabel{globalex}
In this section we shall give the theoretical answer to the
GE-problem by explicitly constructing a
global non-abelian cohomological type object which will
parameterize the classifying set ${\rm Gext} \, (\mathfrak{L},
\, E)$, for a given Leibniz algebra $\mathfrak{L}$, a vector space
$E$ and an epimorphism $\pi: E \to \mathfrak{L}$ of vector spaces
with ${\rm Ker} (\pi) = \mathfrak{g}$.

Let $\Lambda (\mathfrak{L}, \, \mathfrak{g}) =
\bigl(\triangleleft, \, \triangleright, \, f, \, [-, \,
-]_{\mathfrak{g}} \bigl)$ be a crossed system of $\mathfrak{L}$ by
$\mathfrak{g}$ and $\mathfrak{g} \, \# \, \mathfrak{L}$ the
associated crossed product. Then the canonical projection
$\pi_\mathfrak{L} : \mathfrak{g} \, \# \, \mathfrak{L} \to
\mathfrak{L}$, $\pi_\mathfrak{L} (g, x) = x$ is a morphism of
Leibniz algebras and ${\rm Ker} (\pi_{\mathfrak{L}}) =
\mathfrak{g} \times 0 \cong \mathfrak{g}$. In particular,
$\mathfrak{g} \, \# \, \mathfrak{L}$ is an extension of the
Leibniz algebra $\mathfrak{L}$ by the Leibniz algebra
$\mathfrak{g} = (\mathfrak{g}, [-,\, -]_{\mathfrak{g}})$ via:
\begin{eqnarray} \eqlabel{extencrosl}
\xymatrix{ 0 \ar[r] & \mathfrak{g} \ar[r]^{i_{\mathfrak{g}}} &
{\mathfrak{g} \, \# \, \mathfrak{L}} \ar[r]^{\pi_{\mathfrak{L}}} &
\mathfrak{L} \ar[r] & 0 }
\end{eqnarray}
where $i_{\mathfrak{g}} (g) = (g, 0)$, for all $g \in
\mathfrak{g}$. Conversely, the crossed product is the tool to
answer the description part of the GE-problem:

\begin{theorem}\thlabel{extensioliecros}
Let $\mathfrak{L}$ be a Leibniz algebra, $E$ a vector space and
$\pi : E \to \mathfrak{L}$ an epimorphism of vector spaces with
$\mathfrak{g} = {\rm Ker} (\pi)$. Let $[-, \, -]_E$ be a Leibniz
algebra structure on $E$ such that $\pi : (E, [-, \, -]_E) \to
\mathfrak{L}$ is a morphism of Leibniz algebras.

Then there exists a crossed system $\Lambda (\mathfrak{L}, \,
\mathfrak{g}) = \bigl(\triangleleft, \, \triangleright, \, f, \,
[-, \, -]_{\mathfrak{g}} \bigl)$ of $\mathfrak{L}$ by
$\mathfrak{g}$ and an isomorphism of Leibniz algebras $\varphi:
\mathfrak{g} \, \# \, \mathfrak{L} \to (E, [-, \, -]_E)$ that
stabilizes $\mathfrak{g}$ and co-stabilizes $\mathfrak{L}$, i.e.
the following diagram
\begin{eqnarray*} \eqlabel{diagramadoi}
\xymatrix {& \mathfrak{g} \ar[r]^{i_{\mathfrak{g}}} \ar[d]_{Id} &
{\mathfrak{g} \, \# \, \mathfrak{L}}
\ar[r]^{\pi_{\mathfrak{L}}} \ar[d]^{\varphi} & \mathfrak{L} \ar[d]^{Id}\\
& \mathfrak{g} \ar[r]^{i} & {E} \ar[r]^{\pi } & \mathfrak{L} }
\end{eqnarray*}
is commutative. In particular, any Leibniz algebra extension of
$\mathfrak{L}$ by $\mathfrak{g}$ is cohomologous to a crossed
product extension of the form \equref{extencrosl}.
\end{theorem}

\begin{proof}
Let $[-, \, -]_E$ be a Leibniz algebra structure on $E$ such that
$\pi : (E, [-, \, -]_E) \to \mathfrak{L}$ is a morphism of Leibniz
algebras. Since $k$ is a field we can pick a $k$-linear section $s
: \mathfrak{L} \to E$ of $\pi$, i.e. $\pi \circ s = {\rm
Id}_{\mathfrak{L}}$. Using $s$ and $[-, \, -]_E$ we define a
pre-crossed datum $\bigl(\triangleleft_s, \, \triangleright_s, \,
f_s, \, [-, \, -]_{\mathfrak{g}} \bigl)$ of $\mathfrak{L}$ by
$\mathfrak{g}$ as follows:
\begin{eqnarray}
&& \triangleleft = \triangleleft_s :  \mathfrak{g} \times
\mathfrak{L} \to \mathfrak{g}, \quad \,\,\,\,\,\,\, g
\triangleleft x :
= [g, \, s(x)]_E \eqlabel{actiunea1}\\
&& \triangleright = \triangleright_{s} \,\,\, : \mathfrak{L}
\times \mathfrak{g}  \to \mathfrak{g}, \,\,\,\,\,\,\,\,\,\, x
\triangleright g := [s(x), \ g]_E \eqlabel{actiunea2}\\
&& f = f_s \,\,\, : \mathfrak{L}  \times \mathfrak{L}  \to
\mathfrak{g}, \,\,\,\, f (x, y) := [s(x),
\,  s(y)]_E - s([x, \, y])  \eqlabel{cocicluls} \\
&& \left[-, -\right]_{\mathfrak{g} } : \mathfrak{g}  \times
\mathfrak{g} \to \mathfrak{g}, \quad \,\, \left[g, \,
h\right]_{\mathfrak{g}} := [g, \, h]_E
\end{eqnarray}
for all $g$, $h \in \mathfrak{g}$ and $x$, $y \in \mathfrak{L}$.
We observe that the above bilinear maps are well-defined since
$\pi : (E, [-, \, -]_E) \to \mathfrak{L}$ is a morphism of Leibniz
algebras and $s$ is a section of $\pi$. We will prove that $\bigl(\triangleleft_s, \,
\triangleright_s, \, f_s, \, [-, \, -]_{\mathfrak{g}} \bigl)$ is a
crossed system of $\mathfrak{L}$ by $\mathfrak{g}$ and the map
$$
\varphi : \mathfrak{g} \, \# \, \mathfrak{L} \to (E, [-, -]_E),
\qquad \varphi (g, \, x) := g + s(x)
$$
is an isomorphism of Leibniz algebras, that stabilizes
$\mathfrak{g}$ and co-stabilizes $\mathfrak{L}$. Instead of
proving the compatibility conditions (CS0)-(CS7), which require a
very long computation, we use the following trick: $\varphi:
\mathfrak{g} \times \mathfrak{L} \to E$ is a linear isomorphism
between the direct product of vector spaces $\mathfrak{g} \times
\mathfrak{L}$ and the Leibniz algebra $(E, [-,\, -]_E)$ with the
inverse given by $\varphi^{-1}(y) := \bigl(y - s(\pi(y)), \,
\pi(y) \bigl)$, for all $y \in E$. Thus, there exists a unique
Leibniz algebra structure on $\mathfrak{g} \times \mathfrak{L}$
such that $\varphi$ is an isomorphism of Leibniz algebras and this
unique bracket on $\mathfrak{g} \times \mathfrak{L}$ is given by
$[(g, x), \,  (h, y)] := \varphi^{-1} \bigl( [\varphi(g, x), \,
\varphi(h, y)]_E \bigl)$, for all $g$, $h \in \mathfrak{g}$ and
$x$, $y\in \mathfrak{L}$. The proof is finished if we prove that
this bracket is the one defined by \equref{brackcrosspr}
associated to the pre-crossed system $\bigl(\triangleleft_s, \,
\triangleright_s, \, f_s, \, [-, \, -]_{\mathfrak{g}} \bigl)$.
Indeed, for any $g$, $h \in \mathfrak{g}$ and $x$, $y\in
\mathfrak{L}$ we have:
\begin{eqnarray*}
[(g, x), \,  (h, y)] &=& \varphi^{-1} \bigl( \, [\varphi(g, x), \,
\varphi(h, y)]_E \, \bigl) \, = \, \varphi^{-1} \bigl( \, [g + s(x), \, h + s(y) ]_E  \, \bigl) \\
&=& \varphi^{-1} \bigl( \, [g, \, h]_E + [g, \, s(y)]_E + [s(x),
\, h]_E + [s(x), \, s(y)]_E \, \bigl) \\
&=& \bigl(\, [g, \, h]_E + [g, \, s(y)]_E + [s(x), \,
h]_E + [s(x), \, s(y)]_E  - s([x, \, y]), \,\, [x, \, y] \, \bigl) \\
&=& \bigl(\, [g, \, h]_{\mathfrak{g}} + x \triangleright h + g
\triangleleft y + f(x, y), \, [x, \, y ] \, \bigl)
\end{eqnarray*}
as needed. The fact that $\varphi$ stabilizes $\mathfrak{g}$ and
co-stabilizes $\mathfrak{L}$ is straightforward.
\end{proof}

\thref{extensioliecros} shows that the answer to the classification
part of the GE-problem can be reduced to the classification of all
crossed products $\mathfrak{g} \# \mathfrak{L}$ associated to all
crossed systems between $\mathfrak{L}$ and $\mathfrak{g}$.
First we need the following technical result:

\begin{lemma}\lelabel{HHH}
Let $\mathfrak{L}$ be a Leibniz algebra, $\Lambda (\mathfrak{L},
\, \mathfrak{g}) = \bigl(\triangleleft, \, \triangleright, \, f,
\, [-, \, -]_{\mathfrak{g}} \bigl)$ and $\Lambda '(\mathfrak{L},
\, \mathfrak{g}) = \bigl(\triangleleft', \, \triangleright', \,
f', \, [-, \, -]^{'}_{\mathfrak{g}} \bigl)$ two crossed systems of
$\mathfrak{L}$ by $\mathfrak{g}$ and $\mathfrak{g} \#
\mathfrak{L}$, respectively  $\mathfrak{g} \# '\mathfrak{L}$ the
corresponding crossed products. Then there exists a bijection
between the set of all morphisms of Leibniz algebras $\psi:
\mathfrak{g} \# \mathfrak{L} \to \mathfrak{g} \# '\mathfrak{L}$
which stabilize $\mathfrak{g}$ and co-stabilize $\mathfrak{L}$ and
the set of all linear maps $r: \mathfrak{L} \to \mathfrak{g}$
satisfying the following compatibilities for all $g$, $h \in
\mathfrak{g}$, $x$, $y \in \mathfrak{L}$:
\begin{eqnarray}
[g, \, h]_{\mathfrak{g}} &=& [g, \, h]_{\mathfrak{g}}'
\eqlabel{Mor1}\\
g \triangleleft x &=& g \triangleleft ' x + [g, \,
r(x)]_{\mathfrak{g}}' \eqlabel{Mor2}\\
x\triangleright g &=& x \triangleright' g + [ r(x), \,
g]_{\mathfrak{g}}' \eqlabel{Mor3}\\
f (x, \, y) &=& f' (x, \, y) + [ r(x), \, r(y) ]_{\mathfrak{g}}' -
r ([x, \, y]) + x \triangleright' r(y) + r(x) \triangleleft' y
\eqlabel{Mor4}
\end{eqnarray}
Under the above bijection the morphism of Leibniz algebras $\psi =
\psi_{r}: \mathfrak{g} \# \mathfrak{L} \to \mathfrak{g} \#'
\mathfrak{L}$ corresponding to $r: \mathfrak{L} \to \mathfrak{g}$
is given by $ \psi(g, x) = (g + r(x), \, x)$, for all $g \in
\mathfrak{g}$ and $x \in \mathfrak{L}$. Moreover, $\psi =
\psi_{r}$ is an isomorphism with the inverse $\psi^{-1}_{r} =
\psi_{-r}$.
\end{lemma}

\begin{proof} A linear map $\psi: \mathfrak{g} \# \mathfrak{L} \to \mathfrak{g} \#'
\mathfrak{L}$ stabilizes $\mathfrak{g}$ and co-stabilizes
$\mathfrak{L}$ if and only if there exists a uniquely determined
linear map $r: \mathfrak{L} \to \mathfrak{g}$ such that $ \psi(g,
x) = (g + r(x), \, x)$, for all $g \in \mathfrak{g}$ and $x \in
\mathfrak{L}$. Let $\psi = \psi_{r}$ be such a linear map. By a
straightforward computation we can prove that $\psi: \mathfrak{g}
\# \mathfrak{L} \to \mathfrak{g} \#' \mathfrak{L}$ is a morphism
of Leibniz algebras if and only if \equref{Mor1}-\equref{Mor4}
hold. For this it is enough to check the compatibility conditions
$\psi \Bigl( \{ (g, x), \, (h, y) \} \Bigl) = \{ \psi(g, \, x), \,
\psi(h, \, y) \}'$ on the set of generators, i.e. for the set
$\{(g, \, 0) ~|~ g \in \mathfrak{g}\} \cup \{(0, \, x) ~|~ x \in
\mathfrak{L} \}$. For, instance we can see that $\psi \Bigl( \{
(g, 0), \, (h, 0) \} \Bigl) = \{ \psi(g, \, 0), \, \psi(h, \, 0)
\}'$ if and only if $[-, \, -]_{\mathfrak{g}} = [-, \,
-]^{'}_{\mathfrak{g}}$, i.e. \equref{Mor1} holds. In the same way
$\psi \Bigl( \{ (0, x), \, (0, y) \} \Bigl) = \{ \psi(0, \, x), \,
\psi(0, \, y) \}'$ if and only if \equref{Mor4} holds. The details
are left to the reader, being straightforward computations.
\end{proof}

\leref{HHH} leads to the following:

\begin{definition}\delabel{echiaa}
Let $\mathfrak{L}$ be a Leibniz algebra and $\mathfrak{g}$ a
vector space. Two crossed systems $\Lambda (\mathfrak{L}, \,
\mathfrak{g}) = \bigl(\triangleleft, \, \triangleright, \, f, \,
[-, \, -]_{\mathfrak{g}} \bigl)$ and $\Lambda '(\mathfrak{L}, \,
\mathfrak{g}) = \bigl(\triangleleft', \, \triangleright', \, f',
\, [-, \, -]^{'}_{\mathfrak{g}} \bigl)$ of $\mathfrak{L}$ by
$\mathfrak{g}$ are called \emph{cohomologous}, and we denote this
by $\Lambda (\mathfrak{L}, \, \mathfrak{g}) \approx \Lambda'
(\mathfrak{L}, \, \mathfrak{g})$, if and only if $[-, \,
-]_{\mathfrak{g}} = [-, \, -]^{'}_{\mathfrak{g}}$ and there exists
a linear map $r: \mathfrak{L} \to \mathfrak{g}$ such that for any
$g \in \mathfrak{g}$ and $x$, $y \in \mathfrak{L}$ we have:
\begin{eqnarray}
g \triangleleft x &=& g \triangleleft ' x + [g, \,
r(x)]_{\mathfrak{g}} \eqlabel{Mor2a}\\
x\triangleright g &=& x \triangleright' g + [ r(x), \,
g]_{\mathfrak{g}} \eqlabel{Mor3a}\\
f (x, \, y) &=& f' (x, \, y) + [ r(x), \, r(y) ]_{\mathfrak{g}} -
r ([x, \, y]) + x \triangleright' r(y) + r(x) \triangleleft' y
\eqlabel{Mor4a}
\end{eqnarray}
\end{definition}

\begin{example}\exlabel{coboundary}
Let $\Lambda (\mathfrak{L}, \, \mathfrak{g}) =
\bigl(\triangleleft, \, \triangleright, \, f, \, [-, \,
-]_{\mathfrak{g}} \bigl)$ be a pre-crossed datum of $\mathfrak{L}$
by $\mathfrak{g}$ such that $\triangleleft$, $\triangleright$, $f$
are all the trivial maps. Then, $\Lambda (\mathfrak{L}, \,
\mathfrak{g}) = \bigl([-, \, -]_{\mathfrak{g}} \bigl)$ is a
crossed system if and only if $(\mathfrak{g}, \, [-, \,
-]_{\mathfrak{g}})$ is a Leibniz algebra. This type of crossed
system will be called a trivial crossed system and the
crossed product associated to it is just $\mathfrak{g} \times
\mathfrak{L}$, the direct product of Leibniz algebras. An
arbitrary crossed system $\Lambda (\mathfrak{L}, \, \mathfrak{g})
= \bigl(\triangleleft, \, \triangleright, \, f, \, [-, \,
-]_{\mathfrak{g}} \bigl)$ is called a \emph{coboundary} if it is
cohomologous with a trivial crossed system. Thus, $\Lambda
(\mathfrak{L}, \, \mathfrak{g}) = \bigl(\triangleleft, \,
\triangleright, \, f, \, [-, \, -]_{\mathfrak{g}} \bigl)$ is a
coboundary if there exists a linear map $r: \mathfrak{L} \to
\mathfrak{g}$ such that $(\triangleleft, \, \triangleright, \, f)$
are implemented by $r$ via the following formulas: $\triangleleft
:= [- , \, r(-)]_{\mathfrak{g}}$, $\triangleright := [ r(-), \,
-]_{\mathfrak{g}}$ and $f (x, \, y) = [ r(x), \, r(y)
]_{\mathfrak{g}} - r ([x, \, y])$, for all $x$, $y \in
\mathfrak{L}$. Any crossed product
$\mathfrak{g} \# \mathfrak{L}$ associated to a coboundary is
isomorphic to the usual direct product $\mathfrak{g} \times
\mathfrak{L}$ of Leibniz algebras.
\end{example}

As a conclusion of the above results the theoretical answer of the
GE-problem follows:

\begin{theorem}\thlabel{main1222}
Let $\mathfrak{L}$ be a Leibniz algebra, $E$ a vector space and
$\pi : E \to \mathfrak{L}$ an epimorphism of vector spaces with
$\mathfrak{g} = {\rm Ker} (\pi)$. Then $\approx$ is an equivalence
relation on the set ${\mathcal C} {\mathcal S} (\mathfrak{L}, \,
\mathfrak{g})$ of all crossed systems of $\mathfrak{L}$ by
$\mathfrak{g}$. If we denote by ${\mathbb G} {\mathbb H} {\mathbb
L}^{2} \, (\mathfrak{L}, \, \mathfrak{g}) := {\mathcal C}
{\mathcal S} (\mathfrak{L}, \, \mathfrak{g})/ \approx $, then the
map
$$
{\mathbb G} {\mathbb H} {\mathbb L}^{2} \, (\mathfrak{L}, \,
\mathfrak{g}) \to {\rm Gext} \, (E, \mathfrak{L}), \,\,\,
\overline{\bigl(\triangleleft, \, \triangleright, \, f, \, [-, \,
-]_{\mathfrak{g}} \bigl)} \, \longmapsto \, \mathfrak{g} \,
\#_{\bigl(\triangleleft, \, \triangleright, \, f, \, [-, \,
-]_{\mathfrak{g}} \bigl)} \, \mathfrak{L}
$$
is a bijection between ${\mathbb G} {\mathbb H} {\mathbb L}^{2} \,
(\mathfrak{L}, \, \mathfrak{g})$ and ${\rm Gext} (E,
\mathfrak{L})$.
\end{theorem}

\begin{proof} Follows from \thref{extensioliecros} and
\leref{HHH}.
\end{proof}

The explicit computation of the classifying object ${\mathbb G}
{\mathbb H} {\mathbb L}^{2} \, (\mathfrak{L}, \, \mathfrak{g})$ is
a challenging and very difficult problem. The first key step in
decomposing this object is suggested by the compatibility
condition \equref{Mor1} of \leref{HHH}: it shows that two
different Leibniz algebra structures $[-, \, -]_{\mathfrak{g}}$
and $[-, \, -]^{'}_{\mathfrak{g}}$ on $\mathfrak{g}$ give two
different equivalence classes (orbits) of the relation $\approx$
on ${\mathcal C} {\mathcal S} (\mathfrak{L}, \, \mathfrak{g})$.
Let us fix $[-, \, -]_{\mathfrak{g}}$ a Leibniz algebra structure
on $\mathfrak{g}$ and denote by ${\mathcal C} {\mathcal S}_{[-, \,
-]_{\mathfrak{g}}} (\mathfrak{L}, \, \mathfrak{g})$ the set of all
\emph{local crossed systems} of the Leibniz algebra $\mathfrak{L}$
by the Leibniz algebra $(\mathfrak{g}, \, [-, \,
-]_{\mathfrak{g}})$, i.e. the set of all triples
$\bigl(\triangleleft, \, \triangleright, \, f \bigl)$ such that
$\bigl(\triangleleft, \, \triangleright, \, f, \, [-, \,
-]_{\mathfrak{g}} \bigl) \in {\mathcal C} {\mathcal S}
(\mathfrak{L}, \, \mathfrak{g})$. Two local crossed systems
$\bigl(\triangleleft, \, \triangleright, \, f \bigl)$ and
$\bigl(\triangleleft', \, \triangleright', \, f' \bigl)$ are
\emph{local cohomologous} and we denote this by
$\bigl(\triangleleft, \, \triangleright, \, f \bigl) \approx_l
\bigl(\triangleleft', \, \triangleright', \, f' \bigl)$ if there
exists a linear map $r: \mathfrak{L} \to \mathfrak{g}$ satisfying
the compatibility conditions \equref{Mor2a}-\equref{Mor4a}. Then
$\approx_l$ is an equivalent relation on the set ${\mathcal C}
{\mathcal S}_{[-, \, -]_{\mathfrak{g}}} (\mathfrak{L}, \,
\mathfrak{g})$ of all local crossed systems and we denote by
${\mathbb H} {\mathbb L}^{2} \, \bigl(\mathfrak{L}, \,
(\mathfrak{g}, \, [-, \, -]_{\mathfrak{g}} )\bigl)$ the quotient
set ${\mathcal C} {\mathcal S}_{[-, \, -]_{\mathfrak{g}}}
(\mathfrak{L}, \, \mathfrak{g})/ \approx_l$. The object ${\mathbb
H} {\mathbb L}^{2} \, \bigl(\mathfrak{L}, \, (\mathfrak{g}, \, [-,
\, -]_{\mathfrak{g}} )\bigl)$ was constructed in \cite[Corollary
4.2]{am-2013c} as a classifying object of all (non-abelian)
extensions of the Leibniz algebra $\mathfrak{L}$ by a fixed
Leibniz algebra $(\mathfrak{g}, \, [-, \, -]_{\mathfrak{g}})$. The
above considerations highlight the relation between the global
extension problem and the classical extension problem and
give the following decomposition of ${\mathbb G} {\mathbb H}
{\mathbb L}^{2} \, (\mathfrak{L}, \, \mathfrak{g})$.

\begin{corollary} \colabel{desccompcon}
Let $\mathfrak{L}$ be a Leibniz algebra, $E$ a vector space and
$\pi : E \to \mathfrak{L}$ an epimorphism of vector spaces with
$\mathfrak{g} = {\rm Ker} (\pi)$. Then
\begin{equation}\eqlabel{balsoi}
{\mathbb G} {\mathbb H} {\mathbb L}^{2} \, (\mathfrak{L}, \,
\mathfrak{g}) = \, \sqcup_{[-, \, -]_{\mathfrak{g}}} \, {\mathbb
H} {\mathbb L}^{2} \, \bigl(\mathfrak{L}, \, (\mathfrak{g}, \, [-,
\, -]_{\mathfrak{g}} )\bigl)
\end{equation}
where the coproduct on the right hand side is in the category
of sets over all possible Leibniz algebra structures $[-, \,
-]_{\mathfrak{g}}$ on the vector space $\mathfrak{g}$.
\end{corollary}

\subsection*{The abelian case. The computation of ${\mathbb H} {\mathbb L}^{2} \, \bigl(\mathfrak{L}, \,
\mathfrak{g}_0 \bigl)$. }

The formula \equref{balsoi} highlights the complexity of computing
${\mathbb G} {\mathbb H} {\mathbb L}^{2} \, (\mathfrak{L}, \,
\mathfrak{g})$. The first big obstacle is given by the description
of all Leibinz algebra structures on $\mathfrak{g} = {\rm Ker}
(\pi)$. In case that ${\rm dim}_k (\mathfrak{g})$ is large,
then the mission is hopeless. Furthermore, even for a fixed
Leibniz algebra structure $(\mathfrak{g}, [-, -]_{\mathfrak{g}})$,
the explicit computation of the component ${\mathbb H} {\mathbb
L}^{2} \, \bigl(\mathfrak{L}, \, (\mathfrak{g}, \, [-, \,
-]_{\mathfrak{g}} )\bigl)$ from the right hand side of
\equref{balsoi} is far from being an easy problem, since it classifies all extensions of $\mathfrak{L}$
by the Leibinz algebra $(\mathfrak{g}, [-, -]_{\mathfrak{g}})$.
Among all components of the coproduct of \equref{balsoi} the
simplest and best understood is the one corresponding to the
abelian Leibniz algebra $\mathfrak{g} := \mathfrak{g}_0$. We shall
denote by ${\mathbb H} {\mathbb L}^{2} \, \bigl(\mathfrak{L}, \,
\mathfrak{g}_0 \bigl)$ and we will prove that it is a coproduct of
all classical second cohomological groups \cite{LoP}.

In order to do this, we recall from \reref{repLPvseu} that we have denoted by
${\mathcal C} {\mathcal S}_{0} (\mathfrak{L}, \, \mathfrak{g}_0)$ the set of all
abelian local crossed systems of $\mathfrak{L}$ by
$\mathfrak{g}_0$, i.e. all triples $\bigl(\triangleleft, \, \triangleright, \, f \bigl)$
such that $\bigl(\triangleleft, \, \triangleright, \, f, \, [-, \, -]_{\mathfrak{g}} :=
0 \bigl)$ is a crossed system - which is equivalent to the fact that $(\mathfrak{g}, \,
\triangleleft, \, \triangleright) \in {}_{\mathfrak{L}}{\mathcal
M}_{\mathfrak{L}}$ is a Leibniz bimodule over $\mathfrak{L}$ and
$f: \mathfrak{L} \times \mathfrak{L} \to \mathfrak{g}$ is a
$2$-cocyle. Now, \deref{echiaa} for abelian local crossed
systems take the following form: two abelian local crossed systems
$\bigl(\triangleleft, \, \triangleright, \, f \bigl)$ and
$\bigl(\triangleleft', \, \triangleright', \, f' \bigl)$ are local
cohomologous and we denote this by $\bigl(\triangleleft, \,
\triangleright, \, f \bigl) \approx_{l, a} \bigl(\triangleleft',
\, \triangleright', \, f' \bigl)$ if and only if $\triangleleft =
\triangleleft'$, $\triangleright = \triangleright'$ and there
exists a linear map $r: \mathfrak{L} \to \mathfrak{g}$ such that
\begin{equation} \eqlabel{lpsiml}
f (x, \, y) = f' (x, \, y) - r ([x, \, y]) + x \triangleright'
r(y) + r(x) \triangleleft' y
\end{equation}
The equalities $\triangleleft = \triangleleft'$ and
$\triangleright = \triangleright'$ show that two different
Leibniz bimodule structures over $\mathfrak{L}$ give different
equivalent classes in the classifying object ${\mathbb H} {\mathbb
L}^{2} \, \bigl(\mathfrak{L}, \, \mathfrak{g}_0 \bigl)$. Thus, for
its computation we can also fix $(\mathfrak{g}, \, \triangleleft,
\, \triangleright)$ a Leibniz bimodule structure over
$\mathfrak{L}$ and consider for it the set $ZL^2_{(\triangleleft,
\, \triangleright)} \, (\mathfrak{L}, \, \mathfrak{g}_0) $ of all
$2$-cocycles: i.e. bilinear maps $f: \mathfrak{L} \times
\mathfrak{L} \to \mathfrak{g}$ satisfying the compatibility
condition (CS3). Two such cocycles $f$ and $f'$ are (local)
cohomologous and we denote this by $f \approx_{l, a} \, f'$ if there
exists a linear map $r: \mathfrak{L} \to \mathfrak{g}$ such that
\equref{lpsiml} holds. Now, $\approx_{l, a}$ is an equivalent
relation on the set $ZL^2_{(\triangleleft, \, \triangleright)} \,
(\mathfrak{L}, \, \mathfrak{g}_0) $ of all $2$-cocycles and the
quotient set $ZL^2_{(\triangleleft, \, \triangleright)} \,
(\mathfrak{L}, \, \mathfrak{g}_0)/  \approx_{l, a} $ is just the
second cohomological group \cite[Proposition
1.9]{LoP} and is denoted by $HL^2_{(\triangleleft, \,
\triangleright)} \, (\mathfrak{L}, \, \mathfrak{g}_0)$. All the
above considerations prove the following:

\begin{corollary}\colabel{cazuabspargere}
Let $\mathfrak{L}$ be a Leibniz algebra and $\mathfrak{g}$ a
vector space viewed with the abelian Leibniz algebra structure
$\mathfrak{g}_0$. Then:
\begin{equation}\eqlabel{balsoi2}
{\mathbb H} {\mathbb L}^{2} \, \bigl(\mathfrak{L}, \,
\mathfrak{g}_0 \bigl) \, = \, \sqcup_{(\triangleleft, \,
\triangleright)} \, HL^2_{(\triangleleft, \, \triangleright)} \,
(\mathfrak{L}, \, \mathfrak{g}_0)
\end{equation}
where the coproduct on the right hand side is in the category
of sets over all possible Leibniz bimodule structures
$(\triangleleft, \, \triangleright)$ on $\mathfrak{g}$.
\end{corollary}

\subsection*{Metabelian Leibniz algebras. The computation of
${\mathbb H} {\mathbb L}^{2} \, \bigl(\mathfrak{L}_0, \,
\mathfrak{g}_0 \bigl)$. }

We recall that a metabelian group is a group that is an extension
of an abelian group by other abelian group. Having this as a
source of inspiration we introduce the following concept:

\begin{definition}\delabel{cotangent}
A Leibniz algebra is called \emph{metabelian} if it is an
extension of an abelian Leibniz algebra by another abelian Leibniz
algebra.
\end{definition}

Using \prref{axiocross} and \thref{extensioliecros} we obtain that
any metabelian Leibniz algebra is isomorphic to a crossed product
$\mathfrak{L}_0 \, \# \mathfrak{g}_0$, for some vector spaces
$\mathfrak{L}$ and $\mathfrak{g}$. Hence, $\mathfrak{L}_0 \, \#
\mathfrak{g}_0$ has the bracket given for any $g$, $h\in
\mathfrak{g}$, $x$, $y \in \mathfrak{L}$ by:
\begin{eqnarray}
\{ (g, x), \, (h, y) \} &:=& \bigl( x \triangleright h +
g\triangleleft y +  f(x, y), \, 0 \bigl) \eqlabel{cotg2}
\end{eqnarray}
for some bilinear maps
$$
\triangleleft : \mathfrak{g} \times \mathfrak{L} \to \mathfrak{g},
\quad \triangleright: \mathfrak{L} \times \mathfrak{g} \to
\mathfrak{g}, \quad f: \mathfrak{L} \times \mathfrak{L} \to
\mathfrak{g}
$$
satisfying the following compatibility conditions for any $g \in
\mathfrak{g}$, $x$, $y$, $z\in \mathfrak{L}$:
\begin{eqnarray}
(g \triangleleft x) \triangleleft y &=& (g \triangleleft y)
\triangleleft x, \quad  x \rhd (g \triangleleft y) = (x \rhd g)
\triangleleft y = - \, x \rhd (y \rhd g) \eqlabel{discret1} \\
x \rhd f(y, \, z) &=& f(x, \, y) \triangleleft z - f(x, \, z)
\triangleleft y \eqlabel{discret1}
\end{eqnarray}
which are the ones remaining from the axioms (CS0)-(CS7) in
this context. Now, \deref{echiaa} applied for the abelian Leibniz
algebras $\mathfrak{L}_0$ and $\mathfrak{g}_0$ are reduced to the
following: two triples $(\triangleleft, \, \triangleright, \, f)$
and $(\triangleleft', \, \triangleright', \, f')$ are cohomologous
and we denote this by $(\triangleleft, \, \triangleright, \, f)
\approx_{m}^a (\triangleleft', \, \triangleright', \, f')$ if and
only if $\triangleleft = \triangleleft'$, $\triangleright =
\triangleright'$ and there exists a linear map $r: \mathfrak{L}
\to \mathfrak{g}$ such that
\begin{equation} \eqlabel{metaab1}
f (x, \, y) = f' (x, \, y) + x \triangleright' r(y) + r(x)
\triangleleft' y
\end{equation}
for all $x$, $y \in \mathfrak{L}$. Thus, the equality
$\triangleleft = \triangleleft'$, $\triangleright =
\triangleright'$ of this equivalent relation shows that in order to
compute the local cohomological object ${\mathbb H} {\mathbb
L}^{2} \, \bigl(\mathfrak{L}_0, \, \mathfrak{g}_0 \bigl)$ we can
also fix a pair $(\triangleleft, \triangleright)$. This leads to
the following definition:

\begin{definition}\delabel{discreterep}
A \emph{discrete representation} of a vector space $\mathfrak{L}$
is a triple $(\mathfrak{g}, \, \triangleleft, \, \triangleright)$
consisting of a vector space $\mathfrak{g}$ and two bilinear maps
$\triangleleft : \mathfrak{g} \times \mathfrak{L} \to
\mathfrak{g}$ and $\triangleright: \mathfrak{L} \times
\mathfrak{g} \to \mathfrak{g}$ such that
\begin{equation}\eqlabel{catdisc}
(g \triangleleft x) \triangleleft y = (g \triangleleft y)
\triangleleft x, \quad  x \rhd (g \triangleleft y) = (x \rhd g)
\triangleleft y = - \, x \rhd (y \rhd g)
\end{equation}
for all $g \in \mathfrak{g}$ and $x$, $y \in \mathfrak{L}$. We
denote by ${}_{\mathfrak{L}}{\mathcal D}{\mathcal
M}_{\mathfrak{L}}$ the category of all discrete representations of
$\mathfrak{L}$ having as morphisms the linear maps which are
compatible with the left and the right action of $\mathfrak{L}$.
\end{definition}

Let $\mathfrak{L}$ be a vector space and $(\mathfrak{g}, \,
\triangleleft, \, \triangleright) \in {}_{\mathfrak{L}}{\mathcal
D}{\mathcal M}_{\mathfrak{L}}$ a fixed discrete representation
of $\mathfrak{L}$. We denote by $DZL^2_{(\triangleleft, \,
\triangleright)} \, (\mathfrak{L}, \, \mathfrak{g})$ the set of
all \emph{discrete $2$-cocycles}, i.e. all bilinear maps $f:
\mathfrak{L} \times \mathfrak{L} \to \mathfrak{g}$ satisfying the
compatibility condition \equref{discret1}. Let
$DHL^2_{(\triangleleft, \, \triangleright)} \, (\mathfrak{L}, \,
\mathfrak{g})$ be the quotient $DZL^2_{(\triangleleft, \,
\triangleright)} \, (\mathfrak{L}, \, \mathfrak{g})/
\approx_{m}^a$ and we call it the discrete cohomological group. We
have obtained the following:

\begin{corollary}\colabel{cazuabspargere}
Let $\mathfrak{L}$ and $\mathfrak{g}$ be two vector spaces viewed
with the abelian Leibniz algebra structures. Then:
\begin{equation}\eqlabel{balsoi5}
{\mathbb H} {\mathbb L}^{2} \, \bigl(\mathfrak{L}_0, \,
\mathfrak{g}_0 \bigl) \, = \, \sqcup_{(\triangleleft, \,
\triangleright)} \, DHL^2_{(\triangleleft, \, \triangleright)} \,
(\mathfrak{L}, \, \mathfrak{g})
\end{equation}
where the coproduct on the right hand side is in the category
of sets over all pairs $(\triangleleft, \, \triangleright)$ such
that $(\mathfrak{g}, \, \triangleleft, \, \triangleright) \in
{}_{\mathfrak{L}}{\mathcal D}{\mathcal M}_{\mathfrak{L}}$.
\end{corollary}

\coref{cazuabspargere} reduces the computation of ${\mathbb H}
{\mathbb L}^{2} \, \bigl(\mathfrak{L}_0, \, \mathfrak{g}_0 \bigl)$
to a pure linear algebra problem. We shall give an explicit
example below: first we describe all crossed products $k^n_0 \,
\# \, k_0$ and then we shall compute ${\mathbb H} {\mathbb L}^{2}
\, \bigl(k_0, \, k^n_0 \bigl)$. The crossed products $k^n_0 \, \#
\, k_0$ will be parameterized by a very interesting set of
matrices, which we denote by ${\mathcal T} (n)$, and which consist of
all triples $(A, \, B, \, \gamma) \in {\rm M}_n (k) \times {\rm
M}_n (k) \times k^n$ satisfying the following compatibilities:
\begin{equation}\eqlabel{meta13}
AB = BA = - B^2, \qquad B \gamma = 0
\end{equation}

\begin{corollary}\colabel{meta10}
Let $n$ be a positive integer. Then any crossed product $k^n_0 \, \# \, k_0$ is isomorphic to a
Leibniz algebra denoted by $k^{n+1}_{(A, \, B, \, \gamma)}$ which has the basis $\{E_1, \cdots,  E_{n+1}\}$ and
the bracket defined by:
\begin{eqnarray*}
\{E_i, \, E_{n+1} \} := \sum_{j=1}^n \, a_{ji} \, E_j, \quad
\{E_{n+1}, \, E_{i} \} := \sum_{j=1}^n \, b_{ji} \, E_j, \quad
\{E_{n+1}, \, E_{n+1} \} := \sum_{j=1}^n \, \gamma_{j} \, E_j \eqlabel{matan+1}
\end{eqnarray*}
for all $i = 1, \cdots, n$ and $(A = (a_{ij}), \, B = (b_{ij}), \,  \gamma = (\gamma_i) \,) \in {\mathcal T} (n)$. Furthermore,
${\mathbb H} {\mathbb L}^{2} \, \bigl(k_0, \,
k^n_0 \bigl) \, \cong {\mathcal T} (n)/ \approx$, where $\approx$ is the following relation
on ${\mathcal T} (n)$: $(A, \, B, \,  \gamma) \approx (A', \, B', \,  \gamma') $ if and only if $A = A'$, $B = B'$
and there exist $r \in k^n$ such that $\gamma = \gamma' + (B+A) r$.
\end{corollary}

\begin{proof}
Any bilinear map $\triangleleft : k^n \times k \to k^n$ (resp.
$\triangleright: k \times k^n \to k^n$) is uniquely determined by a
linear map $\lambda: k^n \to k^n$ (resp. $\Lambda : k^n \to k^n$)
by the formula $x\triangleleft a = a \lambda (x)$ (resp. $a
\triangleright x = a \Lambda (x)$), for all $a \in k$, $x \in k^n$.
We denote by $A$ (resp. $B$) the matrix associated to $\lambda$
(resp. $\Lambda$) in the canonical basis $\{e_1, \cdots, e_n\}$ of $k^n$. In the same way
any bilinear map $f : k \times k \to k^n$ is uniquely determined by
an element $\gamma \in k^n$ via the formula $f (a, b) = ab \,
\gamma$. Now, we can easily see that the system $(\triangleleft =
\triangleleft_{\lambda}, \, \triangleright
=\triangleright_{\Lambda}, \, f = f_{\gamma}, \, [-, \, -]_{k^n_0}
:= 0)$ is a crossed system if and only if the triple $(A, B,
\gamma)$ satisfies the compatibilities \equref{meta13}, that is $(A,
B, \gamma) \in {\mathcal T} (n)$. The bracket on the Leibniz
algebra $k^{n+1}_{(A, \, B, \, \gamma)}$ is exactly the one given by
\equref{cotg2}, for the crossed product $k^n_0 \, \# \, k_0$
associated to the triple $(\triangleleft =
\triangleleft_{\lambda}, \, \triangleright
=\triangleright_{\Lambda}, \, f = f_{\gamma})$. We have taken as a basis in $k^n_0 \, \# \, k_0$ the
canonical one: $E_1 = (e_1, 0), \cdots $, $E_n = (e_n, 0)$ and $E_{n+1} = (0, 1)$.
Finally, the equivalence relation on ${\mathcal T} (n)$
is precisely the one of given by \equref{metaab1}, written
equivalently for $(A, \, B, \, \gamma)$.
\end{proof}

\section{Co-flag Leibniz algebras. Examples.} \selabel{coflag}
In this section we provide a way to compute ${\mathbb G} {\mathbb
H} {\mathbb L}^{2} \, (\mathfrak{L}, \, \mathfrak{g})$ using a
recursive algorithm for the class of Leibniz algebras as defined
below:

\begin{definition} \delabel{coflaglbz}
Let $\mathfrak{L}$ be a Leibniz algebra and $E$ a vector space. A
Leibniz algebra structure $[- , \, -]_E$ on $E$ is called a
\emph{co-flag Leibniz algebra over $\mathfrak{L}$} of if there
exists a positive integer $n$ and a finite chain of epimorphisms
of Leibniz algebras
\begin{equation} \eqlabel{lant}
\mathfrak{L}_n : = (E, [- , \, -]_E)
\stackrel{\pi_{n}}{\longrightarrow} \mathfrak{L}_{n-1}
\stackrel{\pi_{n-1}}{\longrightarrow} \mathfrak{L}_{n-2} \, \cdots
\, \stackrel{\pi_{2}}{\longrightarrow} \mathfrak{L}_1
\stackrel{\pi_{1}} {\longrightarrow} \mathfrak{L}_{0} :=
\mathfrak{L} {\longrightarrow} 0
\end{equation}
such that ${\rm dim}_k ( {\rm Ker} (\pi_{i}) ) = 1$, for all $i =
1, \cdots, n$. A finite dimensional Leibniz algebra is called a
\emph{co-flag Leibniz algebra} if it is a co-flag Leibniz algebra
over $\{0\}$.
\end{definition}

All co-flag Leibniz structures over $\mathfrak{L}$
can be completely described by a recursive reasoning where the key
step is the case $n =1$: it describes, only in terms of some
elements associated to $\mathfrak{L}$, all Leibniz algebras of
dimension $1 + {\rm dim}_k (\mathfrak{L})$ that have a Leibniz
projection on $\mathfrak{L}$. Then we shall explicitly compute
${\mathbb G} {\mathbb H} {\mathbb L}^{2} \, (\mathfrak{L}, \, k)$.
The tool we use is the following concept:

\begin{definition} \delabel{coflag}
Let $\mathfrak{L}$ be a Leibniz algebra. A \emph{co-flag datum of
$\mathfrak{L}$} is a triple $(\lambda, \, \Lambda, \, f)$, where
$\lambda$, $\Lambda : \mathfrak{L} \to k$ are linear maps, $f :
\mathfrak{L} \times \mathfrak{L} \to k$ is a bilinear map such
that:
\begin{eqnarray*}
&& \lambda ([x, \, y]) = \Lambda ([x, \, y]) = 0, \qquad \Lambda
(x) \Lambda (y) = - \, \Lambda (x) \lambda (y) \eqlabel{coflag1}
\\
&& f([x, \, y], \, z) - f([x, \, z], \, y) - f(x, \, [y, \, z]) =
\Lambda (x) f (y, \, z) - \lambda (z) f(x, \, y) + \lambda (y)
f(x, \, z) \eqlabel{coflag2}
\end{eqnarray*}
for all $x$, $y$, $z\in \mathfrak{L}$. We denote by ${\mathcal C}
{\mathcal F} \, (\mathfrak{L})$ the set of all co-flag data of
$\mathfrak{L}$.
\end{definition}

If $\mathfrak{L}$ is perfect (i.e. $[\mathfrak{L}, \,
\mathfrak{L}] = \mathfrak{L}$), then ${\mathcal C} {\mathcal F} \,
(\mathfrak{L})$ is identical with the set $ZL^2 (\mathfrak{L}, k)$
of all abelian $2$-cocycles, i.e. all bilinear maps $f :
\mathfrak{L} \times \mathfrak{L} \to k$ satisfying the
compatibility condition for any $x$, $y$, $z\in \mathfrak{L}$:
$$
f([x, \, y], \, z) - f([x, \, z], \, y) - f(x, \, [y, \, z]) = 0
$$
The next proposition shows that the set of co-flag data
${\mathcal C} {\mathcal F} \, (\mathfrak{L})$ parameterizes the set
of all crossed systems of $\mathfrak{L}$ by a $1$-dimensional
vector space $\mathfrak{g}$. It also describes the first Leibniz
algebras $\mathfrak{L}_1$ from the exact sequence \equref{lant}.

\begin{proposition}\prlabel{coflagdim1}
Let $\mathfrak{L}$ be a Leibniz algebra and $\mathfrak{g}$ a
$1$-dimensional vector space with a basis $\{g_0\}$. Then there
exists a bijection between the set ${\mathcal C} {\mathcal S} \,
(\mathfrak{L}, \mathfrak{g})$ of all crossed systems of
$\mathfrak{L}$ by $\mathfrak{g}$ and the set ${\mathcal C}
{\mathcal F} \, (\mathfrak{L})$ of all co-flag data of
$\mathfrak{L}$ such that the crossed system $\Lambda(\mathfrak{L},
\mathfrak{g}) = \bigl(\triangleleft, \, \triangleright, \,
\tilde{f}, \, [-, \, -]_{\mathfrak{g}} \bigl)$ corresponding to
$(\lambda, \, \Lambda, \, f) \in {\mathcal C} {\mathcal F} \,
(\mathfrak{L})$ is given for any $x$, $y \in \mathfrak{L}$ by:
\begin{equation}\eqlabel{crosfla1}
g_0 \triangleleft x := \lambda (x) g_0, \quad  x \triangleright
g_0 := \Lambda(x) g_0, \quad \tilde{f} (x, y) := f(x, y) g_0,
\quad \left[-, \, -\right]_{\mathfrak{g}} := 0
\end{equation}
The crossed product $k \# \mathfrak{L}$ associated to the crossed
system \equref{crosfla1} for $\mathfrak{g} : = k$ and $g_0 := 1$
is denoted by $\mathfrak{L}_{(\lambda, \, \Lambda, \, f)}$. It has
the bracket given for any $\alpha$, $\beta \in k$ and $x$, $y \in
\mathfrak{L}$ by:
\begin{equation}\eqlabel{cofpas1}
\{ (\alpha,\, x), \, (\beta, \, y) \} = \bigl( \alpha \lambda (y)
+ \beta \Lambda (x) + f(x, \, y), \, [x, \, y] \bigl)
\end{equation}
Any Leibniz algebra that has a Leibniz projection on
$\mathfrak{L}$ with the kernel of dimension $1$ is isomorphic to
$\mathfrak{L}_{(\lambda, \, \Lambda, \, f)}$, for some $(\lambda,
\, \Lambda, \, f) \in {\mathcal C} {\mathcal F} \,
(\mathfrak{L})$.
\end{proposition}

\begin{proof}
We have to compute all crossed systems $\Lambda(\mathfrak{L},
\mathfrak{g}) = \bigl(\triangleleft, \, \triangleright, \,
\tilde{f}, \, [-, \, -]_{\mathfrak{g}} \bigl)$ of $\mathfrak{L}$
by $\mathfrak{g}$, i.e. all bilinear maps satisfying the axioms
(CS0)-(CS7). Since $\mathfrak{g}$ has dimension $1$ and
$(\mathfrak{g}, [-, \, -]_{\mathfrak{g}})$ is a Leibniz algebra we
must have that $ [-, \, -]_{\mathfrak{g}} = 0$ and any bilinear
map of $\Lambda(\mathfrak{L}, \mathfrak{g})$ is uniquely determined
by a system $(\lambda, \, \Lambda, \, f)$ consisting of two linear
maps $\lambda$, $\Lambda: \mathfrak{L} \to k$ and a bilinear map
$f : \mathfrak{L}\times \mathfrak{L} \to k$ via the formulas
\equref{crosfla1}. It is straightforward to see that the axioms
(CS1)-(SC7) hold if and only if the compatibilities of
\deref{coflag} corresponding for $(\lambda, \, \Lambda, \, f)$
hold.
\end{proof}

We shall classify the Leibniz algebras $\mathfrak{L}_{(\lambda, \,
\Lambda, \, f)}$ by computing ${\mathbb G} {\mathbb H} {\mathbb
L}^{2} \, (\mathfrak{L}, \, k)$: this is the first explicit
classification result for the GE-problem and the key step in the
classification of all co-flag Leibniz algebras over
$\mathfrak{L}$.

\begin{theorem}\thlabel{clascoflagl}
Let $\mathfrak{L}$ be a Leibniz algebra, $E$ a vector space and
$\pi : E \to \mathfrak{L}$ an epimorphism of vector spaces such
that ${\rm Ker} (\pi)$ has dimension $1$. Then
$$
{\rm Gext} \, (E, \mathfrak{L}) \cong {\mathbb G} {\mathbb H}
{\mathbb L}^{2} \, (\mathfrak{L}, \, k) \cong {\mathcal C}
{\mathcal F} \, (\mathfrak{L})/ \approx
$$
where $\approx$ is the following equivalent relation on ${\mathcal
C} {\mathcal F} \, (\mathfrak{L})$: $(\lambda, \, \Lambda, \, f)
\approx (\lambda', \, \Lambda', \, f')$ if and only if $\lambda =
\lambda'$, $\Lambda = \Lambda'$ and there exists a linear map $r :
\mathfrak{L} \to k$ such that for any $x$, $y\in \mathfrak{L}$:
\begin{equation}\eqlabel{nebun3}
f(x, \, y) = f' (x, \, y) - r([x, \, y]) + \Lambda (x) r(y) + r(x)
\lambda(y)
\end{equation}
The bijection between ${\mathcal C} {\mathcal F} \,
(\mathfrak{L})/ \approx$ and ${\rm Gext} \, (E, \mathfrak{L})$ is
given by $ \overline{(\lambda, \, \Lambda, \, f)} \mapsto
\mathfrak{L}_{(\lambda, \, \Lambda, \, f)} $.
\end{theorem}

\begin{proof}
It follows from \thref{main1222} and \prref{coflagdim1} once we
observe that the compatibility conditions from \deref{echiaa},
imposed for the crossed system \equref{crosfla1}, take the form
given in the statement of the theorem.
\end{proof}

\begin{remark}\relabel{nebun5}
The way the equivalence relation on ${\mathcal C} {\mathcal
F} \, (\mathfrak{L})$ in \thref{clascoflagl} is defined suggests
the decomposition as a coproduct of ${\mathbb G} {\mathbb H}
{\mathbb L}^{2} \, (\mathfrak{L}, \, k) \cong {\mathcal C}
{\mathcal F} \, (\mathfrak{L})/ \approx$ as follows: we fix a pair
of linear maps $(\lambda, \Lambda) \in {\rm Hom}_k (\mathfrak{L},
\, k) \times {\rm Hom}_k (\mathfrak{L}, \, k)$ satisfying the
compatibility conditions:
\begin{equation}\eqlabel{nebun4}
\lambda ([x, \, y]) = \Lambda ([x, \, y]) = 0, \qquad \Lambda (x)
\Lambda (y) = - \, \Lambda (x) \lambda (y)
\end{equation}
for all $x$, $y \in \mathfrak{L}$. For such a pair $(\lambda,
\Lambda)$ we denote by $ZL^2_{(\lambda, \Lambda)} \,
(\mathfrak{L}, \, k)$ the set of all $(\lambda,
\Lambda)$-cocycles, i.e. bilinear maps $f: \mathfrak{L} \times
\mathfrak{L} \to k$ satisfying the compatibility
$$
f([x, \, y], \, z) - f([x, \, z], \, y) - f(x, \, [y, \, z]) =
\Lambda (x) f (y, \, z) - \lambda (z) f(x, \, y) + \lambda (y)
f(x, \, z)
$$
for all $x$, $y$, $z\in \mathfrak{L}$. Two $(\lambda,
\Lambda)$-cocycles $f$ and $f'$ are local cohomologous and we denote
this by $f \approx_{l, a} f'$ if and only if there exists a linear
map $r : \mathfrak{L} \to k$ such that \equref{nebun3} holds. We
denote by $HL^2_{(\lambda, \Lambda)} \, (\mathfrak{L}, \, k)$ the
quotient set $ZL^2_{(\lambda, \Lambda)} \, (\mathfrak{L}, \,
k)/\approx_{l, a} $. Then we have:
\begin{equation}\eqlabel{nebun6}
{\mathbb G} {\mathbb H} {\mathbb L}^{2} \, (\mathfrak{L}, \, k)
\cong {\mathcal C} {\mathcal F} \, (\mathfrak{L})/ \approx \,\,
\cong \, \sqcup_{(\lambda, \Lambda)} \, HL^2_{(\lambda, \Lambda)}
\, (\mathfrak{L}, \, k)
\end{equation}
where the coproduct is computed over all pairs
$(\lambda, \Lambda) \in {\rm Hom}_k (\mathfrak{L}, \, k) \times
{\rm Hom}_k (\mathfrak{L}, \, k)$ satisfying \equref{nebun4}.
\end{remark}

Next we shall highlight the efficiency of \thref{clascoflagl} in
classifying co-flag algebras over a given Leibniz algebra. First
we recall that \cite[Example 3.11]{am-2013c} describes and classifies
all $4$-dimensional Leibniz algebras that contain a given
$3$-dimensional Leibniz algebra $\mathfrak{L}$ as a subalgebra.
Now we shall look at the dual case: we shall describe and classify
all $4$-dimensional Leibniz algebras that have a projection on the
same Leibniz algebra $\mathfrak{L}$.

\begin{example} \exlabel{calexpext}
Let $\mathfrak{L}$ be the $3$-dimensional Leibniz algebra with the
basis $\{e_{1}, e_{2}, e_{3}\}$ and the bracket defined by: $
[e_{1}, \, e_{3}] = e_{2}$, $[e_{3}, \, e_{3}] = e_{1}$. Then
there are two families of $4$-dimensional co-flag Leibniz algebras
over $\mathfrak{L}$. These are the Leibniz algebras denoted by
$\mathfrak{L}_{(a, b, c, d)}$ and $\mathfrak{L}^u_{(\beta,
\gamma)}$ having the basis $\{f_1, \, f_2, \, f_3, \, f_4 \}$ and
the bracket $\{ -, \, - \}$ given by:
\begin{eqnarray*}
&\mathfrak{L}_{(a, b, c, d)}:& \quad \{f_1, \, f_3 \} = f_2 + b \,
f_4, \,\, \{f_2, \, f_3 \} = c\, f_4, \,\, \{f_3, \, f_3 \} = f_1
+ d\, f_4, \, \{f_4, f_3 \} = a \, f_1 \\
&\mathfrak{L}^u_{(\beta, \gamma)}:& \quad \{f_1, \, f_3 \} = f_2 +
\beta \, f_4, \,\, \{f_2, \, f_3 \} = - \, \{f_3, \, f_2 \} = (u
\beta - u^2 \gamma) \, f_4, \\
&& \quad \{f_3, \, f_1 \} = - u \gamma \, f_4, \,\,\,\, \{f_3, \,
f_3 \} = f_1 + \gamma \, f_4, \,\,\, \{f_3, \, f_4 \} = - \,
\{f_4, \, f_3 \} = u \, f_4
\end{eqnarray*}
for all $a$, $b$, $c$, $d$, $\beta$, $\gamma \in k$ and $u \in
k^*$. Furthermore, ${\mathbb G} {\mathbb H} {\mathbb L}^{2} \, (\mathfrak{L}, \, k)
\cong k^* \, \sqcup \, k \, \sqcup \, k^*$,
and the equivalence classes of all non-cohomologous extensions of
$\mathfrak{L}$ by $k$ are represented by $\mathfrak{L}_{(a, \, 0,
\, 0, \, 0)}$, $\mathfrak{L}_{(0, \, 0, \, c, \, 0)}$, and
$\mathfrak{L}^u_{(0, \, 0)}$, for all $a \in k^*$, $c\in k$ and $u
\in k^*$.

The detailed computations are rather long but straightforward and
can be provided upon request. We just indicate the main steps of
the proof. First of all we compute the set ${\mathcal C}
{\mathcal F} \, (\mathfrak{L})$ of all co-flag data of
$\mathfrak{L}$. It can be shown that ${\mathcal C} {\mathcal F} \,
(\mathfrak{L})$ is a coproduct of two sets ${\mathcal C} {\mathcal
F} \, (\mathfrak{L}) = {\mathcal C} {\mathcal F}_1 \,
(\mathfrak{L}) \, \sqcup \, {\mathcal C} {\mathcal F}_2 \,
(\mathfrak{L})$, where ${\mathcal C} {\mathcal F}_1 \,
(\mathfrak{L}) \cong k^4$ with the flag datum $(\lambda, \Lambda,
f)$ associated to $(a, \, b, \, c, \, d) \in k^4$ given by:
\begin{equation}\eqlabel{primcofl}
\lambda (e_1) = \lambda (e_2) := 0, \,\, \lambda (e_3) := a, \quad
\Lambda : = 0, \quad
\begin{tabular}{c|ccc}
  $f$ & $e_1$ & $e_2$ & $e_3$ \\
  \hline
  $e_1$   & 0 & 0   & $b$   \\
  $e_2$   & 0 & 0   & $c$    \\
  $e_3$   & 0 & 0   & $d$   \\
\end{tabular}
\end{equation}
and ${\mathcal C} {\mathcal F}_2 \, (\mathfrak{L}) \cong k^*
\times k^2$ with the flag datum $(\lambda, \Lambda, f)$ associated
to $(u, \, \beta, \, \gamma) \in k^* \times k^2$ given by:
\begin{eqnarray}
&&\lambda (e_1) = \lambda (e_2) := 0, \,\, \lambda (e_3) := - u,
\quad \Lambda (e_1) = \Lambda (e_2) := 0, \,\, \Lambda (e_3) := u, \eqlabel{primcoflbb}\\
&&
\begin{tabular}{c|ccc}
  $f$ & $e_1$ & $e_2$ & $e_3$ \\
  \hline
  $e_1$   & 0 & 0   & $\beta$   \\
  $e_2$   & 0 & 0   & $u \beta - u^2 \gamma$    \\
  $e_3$   & $-u \gamma$ & $u^2 \gamma - u \beta$   & $\gamma$   \\
\end{tabular}
\end{eqnarray}
The Leibniz algebra $\mathfrak{L}_{(a, b, c, d)}$ (resp.
$\mathfrak{L}^u_{(\beta, \gamma)}$) is the crossed product $k \#
\mathfrak{L} = \mathfrak{L}_{(\lambda, \Lambda, f)}$ with the
bracket \equref{cofpas1} associated to $(\lambda, \Lambda, f) \in
{\mathcal C} {\mathcal F}_1 \, (\mathfrak{L})$ (resp. $(\lambda,
\Lambda, f) \in {\mathcal C} {\mathcal F}_2 \, (\mathfrak{L})$) --
we have taken $f_1 = (0, \, e_1)$, $f_2 = (0, \, e_2)$, $f_3 = (0,
\, e_3)$ and $f_4 = (1, \, 0)$ as a basis in the vector space
$\mathfrak{L}_{(\lambda, \Lambda, f)} = k \times \mathfrak{L}$.
Now, the equivalence relation \equref{nebun3} on the set
${\mathcal C} {\mathcal F} \, (\mathfrak{L})$ written equivalently
on $k^4$ takes the form: $(a, \, b, \, c, \, d) \approx_1 (a', \,
b', \, c', \, d')$ if and only if $a = a'$ and there exists $r_1$,
$r_2$, $r_3 \in k$ such that
$$
b = b' - r_2 + r_1 a, \quad c = c' + r_2 a, \quad d = d' - r_1 +
r_3 a
$$
A system of representatives of $\approx_1$ on $k^4$ is $\{(a, \, 0,
\, 0, \, 0) \, | \, \, a\in k^* \} \cup \{(0, \, 0, \, c, \, 0) \,
| \, \, c\in k \}$. This shows that $k^4/ \approx_1 \cong \, k^*
\sqcup k $. Finally, the same equivalent relation written this
time on $k^* \times k^2$ is: $(u, \, \beta, \, \gamma) \approx_2
(u', \, \beta', \, \gamma')$ if and only if $u = u'$ and there
exists $r_1$, $r_2 \in k$ such that
$$
\beta = \beta' - r_2 - r_1 u, \quad \gamma = \gamma' - r_1
$$
A system of representatives of $\approx_2$ on $k^* \times k^2$ is
$\{(u, \, 0, \, 0) \, | \, \, u\in k^* \}$, i.e. $k^* \times k^2/
\approx_2 \, \cong k^*$. This finishes the proof.
\end{example}

The next result is the dual version of \coref{meta10}: we describe
all crossed products $k_0 \, \# \, k^n_0$ (which are
$n+1$-dimensional metabelian Leibniz algebras that are co-flag
Leibniz algebras over $k^n_0$) and we shall compute ${\mathbb G}
{\mathbb H} {\mathbb L}^{2} \, \bigl(k^n_0, \, k \bigl)$. In order
to do this we introduce two sets that will parameterize all
crossed products $k_0 \, \# \, k^n_0$. Let ${\mathcal C} {\mathcal
F}_1 \, (n)$ be the set of all pairs $(\Lambda, \, f)$ consisting
of a non-trivial linear map $\Lambda: k^n \to k$ and a bilinear
map $f: k^n \times k^n \to k$ satisfying the following
compatibility for any $x$, $y$, $z\in k^n$:
\begin{equation} \eqlabel{cofma1}
\Lambda (x) f(y, \, z) - \Lambda (y) f(x, \, z) + \Lambda (z) f(x, \, y) = 0
\end{equation}
Two pairs $(\Lambda, \, f)$ and $(\Lambda', \, f') \in {\mathcal C} {\mathcal F}_1 \, (n)$ are
cohomologous and we denote this by $(\Lambda, \, f)\approx_1 (\Lambda', \, f')$ if and only if
$\Lambda = \Lambda'$ and there exists a linear map
$r: k^n \to k$ such that $f (x, \, y) = f'(x, \, y) + \Lambda (x) r(y) - r(x) \Lambda(y)$, for all $x$, $y\in k^n$.

Let ${\mathcal C} {\mathcal F}_2 \, (n)$ be the set of all pairs
$(\lambda, \, f)$ consisting of a linear map $\lambda: k^n \to k$
and a bilinear map $f: k^n \times k^n \to k$ satisfying the
following compatibility condition for any  $x$, $y$, $z\in k^n$:
\begin{equation} \eqlabel{cofma2}
\lambda (y) f(x, \, z) = \lambda (z) f(x, \, y)
\end{equation}
Two pairs $(\lambda, \, f)$ and $(\lambda', \, f') \in {\mathcal
C} {\mathcal F}_2 \, (n)$ are cohomologous and we denote this by
$(\Lambda, \, f)\approx_2 (\Lambda', \, f')$ if and only if
$\lambda = \lambda'$ and there exists a linear map $r: k^n \to k$
such that $f (x, \, y) = f'(x, \, y) + \lambda (x) r(y)$, for all
$x$, $y\in k^n$. With these notations we have:

\begin{proposition} \prlabel{metacof1}
Let $n$ be a positive integer. Then any crossed product $k_0 \, \#
\, k^n_0$ is isomorphic to a Leibniz algebra of the form
$k^{n+1}_{(\Lambda, \, f)}$ or $k^{n+1}_{(\lambda, \, f)}$ as
defined below. These algebras have the basis $\{f_1, \, f_2,
\cdots, f_{n+1} \}$ and the bracket $\{ -, \, - \}$ given for any
$i = 1, \cdots, n$ by:
\begin{eqnarray*}
&k^{n+1}_{(\Lambda, \, f)}:& \qquad \{f_i, \, f_{j} \} := f (e_i,
\, e_j) \, f_{n+1}, \quad \{f_i, \, f_{n+1} \} := \Lambda (e_i) \,
f_{n+1}
\end{eqnarray*}
for all $(\Lambda, \, f) \in {\mathcal C} {\mathcal F}_1 \, (n)$. $k^{n+1}_{(\Lambda, \, f)}$ is a Lie algebra.
The Leibniz algebras $k^{n+1}_{(\lambda, \, f)}$ has the bracket defined for any $i = 1, \cdots, n$ by:
\begin{eqnarray*}
&k^{n+1}_{(\lambda, \, f)}:& \qquad \{f_i, \, f_{j} \} := f (e_i,
\, e_j) \, f_{n+1}, \quad \{f_{n+1}, \, f_{i} \} :=  \lambda (e_i)
\, f_{n+1}
\end{eqnarray*}
for all $(\lambda, \, f) \in {\mathcal C} {\mathcal F}_2 \, (n)$. Furthermore,
\begin{equation} \eqlabel{metamare}
{\mathbb G} {\mathbb H} {\mathbb L}^{2} \, \bigl(k^n_0, \, k \bigl) \, \cong \, {\mathcal C} {\mathcal F}_1 \, (n)/\approx_1 \,
\sqcup \,\, {\mathcal C} {\mathcal F}_2 \, (n)/\approx_2
\end{equation}
\end{proposition}

\begin{proof}
In the first step we prove that the set ${\mathcal C} {\mathcal F}
\, (k^n_0)$ of all co-flag data of the abelian Leibniz algebra
$k^n_0$ is the coproduct of ${\mathcal C} {\mathcal F}_1 \, (n)$
and ${\mathcal C} {\mathcal F}_2 \, (n)$. Indeed, we can easily
see that a triple $(\lambda, \, \Lambda, \, f) \in {\mathcal C}
{\mathcal F} \, (k^n_0)$ if and only if
$$
\Lambda (x) \bigl(\lambda (y) + \Lambda(y) \bigl) = 0, \quad
\Lambda (x) f(y, \, z) = \lambda (z) f(x, \, y) - \lambda (y) f(x, \, z)
$$
for all $x$, $y$, $z \in k^n$. The first equation splits the
description of the elements $(\lambda, \Lambda, f) \in {\mathcal
C} {\mathcal F} \, (k^n_0)$ into two cases. The first one
corresponds to  $\Lambda \neq 0$, the non-trivial map: if this
happens, then $(\lambda, \, \Lambda, \, f) \in {\mathcal C}
{\mathcal F} \, (k^n_0)$ if and only if $\lambda = - \Lambda$ and
$(\Lambda, \, f) \in {\mathcal C} {\mathcal F}_1 \, (n)$. In the
second case, if $\Lambda = 0$ then $(\lambda, \, \Lambda := 0, \,
f) \in {\mathcal C} {\mathcal F} \, (k^n_0)$ if and only if
$(\lambda, \, f) \in {\mathcal C} {\mathcal F}_2 \, (n)$. This
shows that ${\mathcal C} {\mathcal F} \, (k^n_0) \, = \, {\mathcal
C} {\mathcal F}_1 \, (n) \, \sqcup \,\, {\mathcal C} {\mathcal
F}_2 \, (n)$.

Let $\{e_1, \cdots,  e_n\}$ be the canonical basis of $k^n$ and
$f_1 := (0, e_1), \cdots$, $f_n := (0, e_n)$, $f_{n+1} := (1, 0)$
the basis in $k \times k^n$. Then the Lie algebra
$k^{n+1}_{(\Lambda, \, f)}$ is the crossed product $k \# k^n_{0}$
defined by \equref{cofpas1} associated to the co-flag datum
$(-\Lambda, \, \Lambda, \, f)$, for all $(\Lambda, \, f) \in
{\mathcal C} {\mathcal F}_1 \, (n)$. The Leibniz algebra
$k^{n+1}_{(\lambda, \, f)}$ is the crossed product $k \# k^n_{0}$
defined by \equref{cofpas1} associated to the co-flag datum
$(\lambda, \, \Lambda := 0, \, f)$, for all $(\lambda, \, f) \in
{\mathcal C} {\mathcal F}_2 \, (n)$. The last statement follows
once we observe that the equivalent relation $\approx$ on the set
${\mathcal C} {\mathcal F} \, (k^n_0)$ from \thref{clascoflagl}
written on its disjoint components ${\mathcal C} {\mathcal F}_1 \,
(n)$ and ${\mathcal C} {\mathcal F}_2 \, (n)$ takes the forms
$\approx_1$ and respectively $\approx_2$ as defined above.
\end{proof}

\prref{metacof1} can be detailed for any given $n$. In particular,
for $n = 1$ we obtain:

\begin{corollary} \colabel{coflagdim2}
Any $2$-dimensional co-flag (or metabelian) Leibniz algebra
is isomorphic to one of the following Leibniz
algebras:
\begin{eqnarray*}
&k^2_{a, c}:& \qquad \{f_1, \, f_2 \} = a  \,
f_1, \quad \{f_2, \, f_2\} = c \, f_1, \qquad a, \, c \in k \\
&k^2_{b}:& \qquad \{f_2, \, f_1\} = - \{f_1, \, f_2\} = b \, f_1,
\qquad b \in k^*
\end{eqnarray*}
Furthermore, ${\mathbb G} {\mathbb H} {\mathbb L}^{2} \, (k_0, \,
k) \cong k \, \sqcup \, k^* \, \sqcup \, k^*$ and the set of
equivalent classes of all non-cohomologous extensions of $k_{0}$
by $k$ is represented by $k^2_{a, 0}$, $k^2_{0, c}$ or $k^2_{b}$,
for all $a\in k$ and $b$, $c\in k^*$.
\end{corollary}

\begin{remark} \relabel{clasificarea2dim}
The three families of Leibniz algebras of \coref{coflagdim2} give
the classification of $2$-dimensional co-flag Leibniz algebras
from the viewpoint of the GE-problem: i.e. up to an isomorphism
of Leibniz algebras that stabilizes $k$ and co-stabilizes $k$. If we
classify these Leibniz algebras only up to an isomorphism then the only non-isomorphic Leibniz algebras are the
following: $k^2_{0} := k^2_{0, 0}$ (the
abelian Leibniz algebra), $k^2_{1, 0}$, $k^2_{0, 1}$ and the Lie
algebra $k^2_1$. These are all Leibniz
algebras of dimension $2$ (\cite{Go}, \cite{Lod2}).
\end{remark}

Applying \prref{metacof1} for $n = 2$ we obtain:

\begin{corollary}\colabel{coflag3.1}
Any $3$-dimensional co-flag Leibniz
algebra over $k^2_{0}$ is cohomologous with one of the Leibniz
algebras of the six families described below. More precisely, we
have:
\begin{equation} \eqlabel{coflag3.1bal}
{\mathbb G} {\mathbb H} {\mathbb L}^{2} \, (k^2_{0}, \, k) \,
\cong \, \sqcup_{i = 1}^6 \, HL^2_{i} \, (k^2_{0}, \, k), \quad {\rm where:}
\end{equation}
$\bullet$ $HL^2_{1} \, (k^2_{0}, \, k) = k^* \times k$ and the corresponding equivalent classes of
non-cohomologous extensions of $k^2_{0}$ by $k$ are the Lie algebras with the bracket defined for any
$(\Lambda_1, \, \Lambda_2) \in k^* \times k$ by:
\begin{eqnarray*}
k^{3, \, 1}_{(\Lambda_1, \Lambda_2)}: \qquad
\{f_1, \, f_3 \} = - \{f_3, \, f_1\} = \Lambda_1 \, f_3, \,\,\,\,
\{f_2, \, f_3 \} = - \{f_3, \, f_2 \} = \Lambda_2 \, f_3
\end{eqnarray*}
$\bullet$ $HL^2_{2} \, (k^2_{0}, \, k) = k^* $ and the corresponding equivalent classes of
non-cohomologous extensions of $k^2_{0}$ by $k$ are the Lie algebras with the bracket defined for any
$\Lambda_2 \in k^*$ by:
\begin{eqnarray*}
k^{3, \, 2}_{\Lambda_2}: \qquad
\{f_2, \, f_3 \} = - \{f_3, \, f_2\} = \Lambda_2 \, f_3
\end{eqnarray*}
$\bullet$ $HL^2_{3} \, (k^2_{0}, \, k) = k^4$ and the corresponding equivalent classes of
non-cohomologous extensions of $k^2_{0}$ by $k$ are the Leibniz algebras with the bracket defined for any
$(a, \, b, \, c, \, d) \in k^4$ by:
\begin{eqnarray*}
k^{3, \, 3}_{(a, b, c, d)}: \qquad
\{f_1, \, f_1 \} = a \, f_3, \,\,\,\, \{f_1, \, f_2 \} = b \, f_3, \,\,\,\,
\{f_2, \, f_1 \} = c \, f_3, \,\,\,\, \{f_2, \, f_2 \} = d \, f_3
\end{eqnarray*}
$\bullet$ $HL^2_{4} \, (k^2_{0}, \, k) = k^* $ and the corresponding equivalent classes of
non-cohomologous extensions of $k^2_{0}$ by $k$ are the Leibniz algebras with the bracket defined for any
$\lambda_2 \in k^*$ by:
\begin{eqnarray*}
k^{3, \, 4}_{\lambda_2}: \qquad \{f_3, \, f_2\} = \lambda_2 \, f_3
\end{eqnarray*}
$\bullet$ $HL^2_{5} \, (k^2_{0}, \, k) = k^* $ and the corresponding equivalent classes of
non-cohomologous extensions of $k^2_{0}$ by $k$ are the Leibniz algebras with the bracket defined for any
$\lambda_1 \in k^*$ by:
\begin{eqnarray*}
k^{3, \, 5}_{\lambda_1}: \qquad \{f_3, \, f_1\} = \lambda_1 \, f_3
\end{eqnarray*}
$\bullet$ $HL^2_{6} \, (k^2_{0}, \, k) = k^* \times k^* $ and the corresponding equivalent classes of
non-cohomologous extensions of $k^2_{0}$ by $k$ are the Leibniz algebras with the bracket defined for any
$(\lambda_1, \lambda_2) \in k^* \times k^*$ by:
\begin{eqnarray*}
k^{3, \, 6}_{(\lambda_1, \lambda_2)}: \qquad \{f_3, \, f_1\} = \lambda_1 \, f_3, \quad \{f_3, \, f_2\} =
\lambda_2 \, f_3
\end{eqnarray*}
\end{corollary}

\begin{proof}
We apply \prref{metacof1} for $n = 2$ detailing the formula
\equref{metamare}. The computations are long but straightforward
and we only indicate the main steps of the proof. A linear map
$\Lambda : k^2 \to k$ (resp. $\lambda : k^2 \to k$) is given by
two scalars $(\Lambda_1, \Lambda_2) \in k^2$ (resp. $(\lambda_1,
\lambda_2) \in k^2$). The first two families of Lie algebras from
the statement correspond to the case $\Lambda \neq 0$: more
precisely, $k^{3, \, 1}_{(\Lambda_1, \, \Lambda_2)}$ is associated
to $\Lambda_1 \neq 0$ and $k^{3, \, 2}_{\Lambda_2}$ is associated
to $\Lambda_1 = 0$ and $\Lambda_2 \neq 0$. We give a few more
details only for the first case: we can easily prove that
$(\Lambda, f) \in {\mathcal C} {\mathcal F}_1 \, (2)$ if and only
if the cocycle $f$ is given by: $f (e_1, \, e_1) = f (e_2, \, e_2)
= 0$ and $f (e_1, \, e_2) = - f (e_2, \, e_1) = a$, for some $a
\in k$. The crossed product $k \# k^2_{0}$ associated to this
co-flag datum is denoted by $k^{3, \, 1}_{(\Lambda_1, \,
\Lambda_2, \, a)}$, for any $(\Lambda_1, \, \Lambda_2, \, a) \in
k^* \times k \times k$. The bracket in $k^{3, \, 1}_{(\Lambda_1,
\, \Lambda_2, \, a)}$ is given by:
$$
\{f_1,  f_2 \} = - \{f_2,  f_1\} = a \, f_3, \,\, \{f_1, f_3 \} =
- \{f_3, f_1\} = \Lambda_1 \, f_3, \,\, \{f_2,  f_3 \} = - \{f_3,
f_2 \} = \Lambda_2 \, f_3
$$
It is straightforward to see that the set $k^* \times k \times
\{0\}$ is a system of representatives of the equivalent relation
$\approx_1$, that is any Lie algebra $k^{3, \, 1}_{(\Lambda_1, \,
\Lambda_2, \, a)}$ is cohomologous to $k^{3, \, 1}_{(\Lambda_1, \,
\Lambda_2, \, 0)}$ and this shows that $HL^2_{1} \, (k^2_{0}, \,
k) = k^* \times k$.

The rest of the proof continues along the same lines. We just
mention that the last four families of Leibniz algebras arise from
the case $\Lambda_1 = \Lambda_2 = 0$ which in turn can be
disjoined in four subcases depending on whether $\lambda_1$ or
$\lambda_2$ is equal to $0$ or not. For instance, the last family
of Leibniz algebras is associated to the subcase $\lambda_1 \neq 0
\neq \lambda_2$.
\end{proof}

Any Leibniz algebra from \coref{coflag3.1} is metabelian, being a
crossed product $k_0 \, \# \, k^2_0$. On the other hand the
crossed products $k^2_0 \, \# \, k$ are described in
\coref{meta10}. Thus we have completed the description of all
$3$-dimensional metabelian Leibniz algebras:

\begin{corollary}\colabel{metadim3clas}
Any $3$-dimensional metabelian Leibniz algebra is isomorphic to
$k^3_0$, $k^3_{(A, B, \gamma)}$, for some $(A, B, \gamma) \in
{\mathcal T} (2)$ or to one of the Leibniz algebras from the six
families listed in \coref{coflag3.1}.
\end{corollary}

Now we shall indicate how the recursive algorithm provided by
\thref{clascoflagl} works for the classification of co-flag
Leibniz algebras. The \coref{coflagdim2} is the first step in the
classification of finite dimensional co-flag Leibniz algebras. In
order to continue the description of all $3$-dimensional co-flag
Leibniz algebras we have to take instead of $\mathfrak{L}$ one of
the Leibniz algebras $k^2_{0}$, $k^2_{1, 0}$, $k^2_{0, 1}$ or
$k^2_1$ and then to compute the set ${\mathcal C} {\mathcal F} \,
(\mathfrak{L})$ of all its co-flag data. Then we have to describe
the associated crossed products $k \# \mathfrak{L} =
\mathfrak{L}_{(\lambda, \Lambda, f)}$, for all $(\lambda, \Lambda,
f) \in {\mathcal C} {\mathcal F} \, (\mathfrak{L})$. Finally, we
have to classify these Leibniz algebras by computing ${\mathbb G}
{\mathbb H} {\mathbb L}^{2} \, (\mathfrak{L}, \, k)$ using the
decomposition given by \equref{nebun6} of \reref{nebun5}.
\coref{coflag3.1} solved the case $\mathfrak{L} := k^2_{0}$, the
abelian Leibniz algebra of dimension $2$: it is relevant to
highlight the size of ${\mathbb G} {\mathbb H} {\mathbb L}^{2} \,
(k^2_{0}, \, k)$. To complete the description of all
$3$-dimensional co-flag Leibniz algebras we just have to repeat
the method of \coref{coflag3.1} by replacing the abelian Leibniz
algebra $k^2_0$ with $k^2_{1, 0}$, $k^2_{0, 1}$ and respectively
$k^2_1$. The question is left to the reader since it is an linear
algebra exercise.

We prefer to end the paper with an interesting example. Even if
all the above examples confirm that, as a general rule, ${\mathbb
G} {\mathbb H} {\mathbb L}^{2} \, (\mathfrak{L}, \, k)$ is a very
large object, we shall give a non-trivial example for which
${\mathbb G} {\mathbb H} {\mathbb L}^{2} \, (\mathfrak{L}, \, k)$
is a singleton. In order to do that, we observe that the formula
\equref{nebun6} is simplified if $\mathfrak{L}$ is a perfect
Leibniz algebra. In this case the only pair $(\lambda, \Lambda)
\in {\rm Hom}_k (\mathfrak{L}, \, k) \times {\rm Hom}_k
(\mathfrak{L}, \, k)$ satisfying the compatibility condition
\equref{nebun4} is the trivial one $(\lambda = \Lambda := 0)$.
Thus, we have that:
\begin{equation}\eqlabel{nebun7}
{\mathbb G} {\mathbb H} {\mathbb L}^{2} \, (\mathfrak{L}, \, k)
\cong {\mathcal C} {\mathcal F} \, (\mathfrak{L})/ \approx \,\,
\cong \, \, HL^2_{(0, 0)} \, (\mathfrak{L}, \, k)
\end{equation}

\begin{example}\exlabel{single}
Let $k$ be a field of characteristic $\neq 2$ and $\mathfrak{L} =
\mathfrak{sl} (2, k)$, the Lie algebra with a basis $\{e_1, \,
e_2, \, e_3\}$ and the usual bracket $[e_1, \, e_2] = e_3$, $[e_1,
\, e_3] = -2 \, e_1$ and $[e_2, \, e_3] = 2 \, e_2$. Then any
$4$-dimensional co-flag Leibniz algebra over $\mathfrak{sl} (2,
k)$ is isomorphic to the Lie algebra with the basis $\{f_1, \,
f_2, \, f_3, \, f_4 \}$ and the bracket given by:
\begin{eqnarray*}
\mathfrak{sl} (2, k)_{(a, b, c)}: \quad \, \{f_1, \, f_2 \} = f_3
+ a \, f_4, \,\,\, \{f_1, \, f_3 \} = -2 f_1 + b\, f_4, \,\,\,
\{f_2, \, f_3 \} = 2 f_2 + c \, f_4
\end{eqnarray*}
for all $a$, $b$, $c\in k$.  Furthermore, ${\mathbb G} {\mathbb H}
{\mathbb L}^{2} \, (\mathfrak{sl} (2, k), \, k) = \{0 \}$, i.e.
any $4$-dimensional Leibniz algebra that has a Leibniz projection on
$\mathfrak{sl} (2, k)$ is cohomologous with the direct product of
Lie algebras $k_0 \times \mathfrak{sl} (2, k)$.

The proof is similar to the one on \exref{calexpext}, but it is
easier since $\mathfrak{sl} (2, k)$ is perfect as a Lie
algebra. Routine computations show that the set ${\mathcal C}
{\mathcal F} \, (\mathfrak{sl} (2, k))$ of all co-flag data of
$\mathfrak{sl} (2, k)$ identify with $k^3$: the bijection is given such that
the co-flag datum $(\lambda, \Lambda, f)$ associated to $(a, \, b, \, c) \in k^3$ is
defined by:
\begin{equation}\eqlabel{panamea}
\lambda = \Lambda : = 0,  \qquad
\begin{tabular}{c|ccc}
  $f$ & $e_1$ & $e_2$ & $e_3$ \\
  \hline
  $e_1$   & 0 & $a$   & $b$   \\
  $e_2$   & $-a$ & 0   & $c$    \\
  $e_3$   & $-b$ & $-c$   & $0$   \\
\end{tabular}
\end{equation}
The Lie algebra $\mathfrak{sl} (2, k)_{(a, b, c)}$ is the
crossed product \equref{cofpas1} associated to this flag datum.
The equivalence relation \equref{nebun3} written on
$k^3$ becomes: $(a, \, b, \, c) \approx (a', \, b', \, c')$ if and
only if there exists $r_1$, $r_2$, $r_3 \in k$ such that
$$
a = a' - r_3, \quad b = b' - 2 r_1, \quad c = c' - 2 r_2
$$
Thus, the quotient set $k^3/\approx$ is a singleton having
$\overline{(0, \, 0, \, 0)}$ as unique element and we are done.
\end{example}

\end{document}